\documentclass[a4paper,12pt]{article}
\usepackage[total={6.5in, 8.5in}]{geometry}

\usepackage{amssymb, amsfonts, amsthm, amsmath, nicefrac}
\numberwithin{equation}{section}
\usepackage[utf8]{inputenc} 			
\usepackage{subfig} 					

\usepackage[round]{natbib}              
\usepackage{lmodern}
\usepackage{tikz} 						
\usepackage{dsfont} 					
\usepackage{microtype} 					
\usepackage[textsize=scriptsize]{todonotes}

\usepackage{graphicx, float}

\usepackage[inline]{enumitem}           
\setlist[enumerate,1]{label=(\roman*)}  

\usepackage[bookmarks=true, pdfpagemode= UseNone, backref= page, 
pdftitle={time consistency}, pdfauthor={Alexander Shapiro, Alois Pichler}, colorlinks=true, citecolor=blue, urlcolor=blue, pdfstartview= FitH]{hyperref}
\usepackage{doi}

\usepackage{graphicx, float}
\graphicspath{{figures/}}

\newcommand{\evar}{{\sf EV@R}}

\newcommand{\cA}{{\mathfrak A}}

\newcommand{\cX}{{\mathfrak X}}

\newcommand{\cP}{{\mathfrak P}}

\newcommand{\cM}{{\mathfrak M}}
\newcommand{\cY}{{\mathfrak Y}}

\newcommand{\cG}{{\mathfrak G}}

\newcommand{\cF}{{\mathfrak F}}

\newcommand{\cR}{{\mathfrak R}}
\newcommand{\cT}{{\mathfrak T}}

\newcommand{\cv}{{\mathfrak v}}

\theoremstyle{plain}
\newtheorem{theorem}             {Theorem}[section]

\newtheorem{proposition}[theorem]{Proposition}
\theoremstyle{definition}
\newtheorem{definition} [theorem]{Definition}
\theoremstyle{remark}
\newtheorem{example}[theorem]{Example}
\newtheorem{remark}[theorem]{Remark}

\newtheorem{assum}{Assumption}[section]

\newcommand{\x}{{ x}}

\newcommand{\N}{{\cal N}}

\newcommand{\V}{{\cal V}}
\newcommand{\Z}{{\cal Z}}
\newcommand{\F}{{\cal F}}

\newcommand{\X}{{\cal X}}

\newcommand{\R}{\mathcal{R}}

\newcommand{\be}{\begin{equation}}
\newcommand{\ee}{\end{equation}}

\newcommand{\rr}{\rightrightarrows}

\newcommand{\bx}{\mbox{\boldmath$x$}}

\newcommand{\bxx}{\mbox{\scriptsize\boldmath$\x$}}

\newcommand{\bsx}{\mbox{\boldmath\scriptsize$x$}}

\newcommand{\lan}{\langle}
\newcommand{\ran}{\rangle}

\def\w{\omega}
\def\e{\varepsilon}
\def\O{\Omega}

\DeclareMathOperator*{\argmin}{arg\,min}
\DeclareMathOperator*{\argmax}{arg\,max}
\DeclareMathOperator*{\essinf}{ess\,inf}
\DeclareMathOperator*{\ess}{ess\,sup}
\DeclareMathOperator*{\esssup}{ess\,sup}
\DeclareMathOperator*{\minmax}{min/\,max}

\DeclareMathOperator{\one}{{\mathds 1}} 	

\DeclareMathOperator*{\Min}{Min}

\def\se {{\,\succeq\,}}

\newcommand{\avr} {\AVaR}
\DeclareMathOperator{\AVaR}{\mathsf{AV@R}}	
\DeclareMathOperator{\EVaR}{\mathsf {EV@R}}

\def\bbr{{\Bbb{R}}} 

\def\bbe{{\Bbb{E}}}

\let\oldproofname=\proofname
\renewcommand{\proofname}{\rm\bf{\oldproofname}}

\begin{document}
\title{\textbf{Risk averse stochastic programming: time consistency and optimal stopping}}

\author{
{\bf Alois Pichler}
\\Fakult{\"a}t für Mathematik\\Technische Universit{\"a}t Chemnitz\\
D--09111 Chemnitz, Germany
\and
{\bf  Rui Peng Liu} and  {\bf Alexander  Shapiro}\thanks{
Research of this author  was partly supported by NSF grant 1633196, and  DARPA EQUiPS program  grant SNL 014150709.}\\
School of Industrial and Systems Engineering\\
Georgia Institute of Technology\\
Atlanta, GA 30332-0205\\
}
\maketitle
\thispagestyle{empty}

\begin{abstract}
	This paper addresses time consistency of risk averse optimal stopping in stochastic optimization. It is demonstrated that time consistent optimal stopping entails a specific structure of the functionals describing the transition between consecutive stages.
	The stopping risk measures capture this structural behavior and they allow natural dynamic equations for risk averse decision making over time.
	Consequently, associated optimal policies satisfy Bellman's principle of optimality, which characterizes optimal policies for optimization by stating that a decision maker should not reconsider
previous decisions retrospectively. We also discuss numerical approaches to solving such problems.
\end{abstract}

\noindent
{\bf Keywords:} Stochastic programming, coherent risk measures, time consistency, dynamic equations, optimal stopping time, Snell envelope.

\section{Introduction}\label{sec:intr}
Optimal Stopping (OS)   is a classical topic of research  in statistics and operations research going back to the pioneering work of \citet{Wald1947, Wald1949} on sequential analysis.   For a thorough discussion of theoretical foundations of OS we can refer to  \cite{ShiryaevStopping}, for example. The classical formulation of OS assumes availability of the probability distribution of the considered data process. Of course in real life applications the `true distribution' is never known exactly and in itself can be viewed as uncertain. This motivates to formulate a considered   optimization problem in terms of a chosen family of probability distributions. In stochastic optimization this is going back to the pioneering paper of
\cite{Scarf}, and   is often referred to as the distributionally robust approach to stochastic programming. In a certain sense  the distributional robustness can be considered as a dual counterpart
of  the modern approach to   risk measures, cf.\ \cite{artzner1999}.
Recently distributionally robust (risk averse)
formulations of OS were considered in several publications, e.g., \citet[Section~6.5]{Follmer2004}, \cite{cher2006},  \cite{Karatzas2010}, \cite{Kraetschmer}, \cite{Kraetschmer2016}, \cite{Belom2017}, \cite{Goldenshluger}.

However, a straightforward extension of risk neutral stochastic programs to their distributionally robust (risk averse)  counterparts involves a delicate issue of time consistency.
 That is, decisions (optimal policies)  designed at the initial time,    before observing  any realization of the random data process,  should not be reconsidered  from a   perspective of later stages.
This is closely  related to Bellman’s principle of optimality and to eventual  writing of  the corresponding  dynamic programming equations which is crucial for an efficient numerical solution  of such problems. The question  of time (in)consistency  has been realized very early in economics and management
science.
\cite{Weller1978} (cf.\ also \cite{Hammond1989}) discusses time consistency in combination with expected utilities. Quoting \citet[p.~263]{Weller1978}:
		``I~prove that under certain assumptions, consistency is equivalent to maximizing expected utility on the set of feasible plans, with a restricted set of utility functions and a tree of subjective probability distributions which satisfy the Bayesian updating rule.''
This problem setting puts risk measures in relation to stochastic optimization. In this context, however, it is to mention that 	\cite{Haviv1996} gives a Markov decision process with constraints, so that the optimal solution does \emph{not} satisfy Bellman's principle.

More recently
the discussion of consistency properties of risk measures has become popular in financial mathematics. Pioneers include \cite{Wang1999} and  \cite{Jobert}, who introduce a concept of dynamic consistency of risk measures themselves. Other authors (e.g., \cite{Weber2006}, \cite{Cheridito2011}, \cite{Ruszczynski2010}) take a similar approach by discussing axioms of risk measures. These publications try to identify properties of risk measures themselves, which are relevant in a general, time consistent framework. Time consistency also respects increasing information, which can be investigated as essential feature of conditional risk measures  (cf.\ \cite{Kovacevic2009}).
\cite{Schachermayer2009a} show  that time consistent and law invariant risk measures have a very specific, entropy like representation which depends on not more than a single parameter.

Optimal stopping together with (risk neutral) policy optimization was considered in  \cite{hordijk}, but was not developed further. It turns out that optimal stopping problems lead to dynamic equations in the classical (risk neutral)  setting.
In this paper we extend this in several directions.
We combine stopping times and risk measures by introducing \emph{stopping risk measures} and demonstrate how optimal stopping times combine with optimizing over policies (decision rules). The appropriate nested risk measures corresponding to the problem structure entail transition functionals at each stage. Their specific structure is employed to specify dynamic equations and generalized risk martingales. We associated dynamic equations corresponding to optimal stopping and investigate their specific structure. Optimal policies derived from optimal stopping problems are naturally time consistent.

We also develop a concept of time consistency in the distributionally robust (risk averse) setting  which is naturally  amenable to Bellman’s principle of optimality and to formulation  of the dynamic programming equations.
These equations are similar to dynamic equations in \citet[Theorem 6.52]{Follmer2004} but do not require the restrictive assumption of stability used there (see Remark~\ref{rem-stab} below).
Finally we discuss computational approaches to solving such problems and give numerical examples.

\paragraph{Outline.} The paper is organized as follows. In Section \ref{sec-optst} we  formulate the classical  (risk neutral) OS problem in discrete time.
Time and dynamic consistency of a general class of multistage optimization problems is discussed in Section \ref{sec-timecon}. In Section \ref{sec:decomp} the respective dynamic programming equations, based on   decomposability of the nested formulation, are derived. Time consistent formulations of the  distributionally robust (risk averse) OS problems are presented in Section \ref{sec:Stopping}. Finally Section \ref{sec:numeric} is devoted to computational methods and numerical examples.

\section{Optimal stopping time}
\label{sec-optst}

We will use the following framework. Let $(\O,\F,P)$ be a probability space and $\mathfrak F:=(\F_0,\dots,\F_T)$ be a
 filtration (a sequence of increasing sigma algebras, $\F_0\subset\cdots\subset\F_{T}$)
with\footnote{The sigma algebra consisting of only two sets, the empty set and the whole space $\O$, is called trivial.}
 $\F_0=\{\emptyset,\,\O\}$ and $\F_T=\F$.
 Let $\Z_0\subset \cdots\subset \Z_T$ be a sequence of
linear spaces of  functions   $Z\colon\O\to\bbr$.
We assume that  $\Z_t:=L_p(\O,\F_t,P)$  for some $p\in [1,\infty]$, although more general settings are possible.
We   denote by  $Z_t$ an element of the space $\Z_t$. Note   that an element  $Z_t\in \Z_t$ actually is a class of $\F_t$-measurable functions which can be different from each other on  a set of $P$-measure zero.
Since $\F_0$ is trivial, the space $\Z_0$ consists of constant functions and will be identified with $\bbr$.
Since elements of  $\Z_t$ are  $\F_t$-measurable,  a  process $Z_t\in \Z_t$, $t=0,\dots,T$,  is adapted  to the filtration $\cF$.  We use the notation
$	\Z_{t,u}:=\Z_t\times \cdots\times \Z_u,$   $0\le t<u\le T,$
in particular $\Z_{0,T}=\Z_0\times \cdots\times \Z_T$.
For elements  $Z_{t,u}=(Z_t,\dots,Z_u)$ and $Z'_{t,u}=(Z'_t,\dots,Z'_u)$ of  $\Z_{t,u}$ we write
$Z_{t,u}\preceq Z'_{t,u}$ to denote that $Z_\tau(\w)\le Z_\tau'(\w)$
for   almost every  (a.e.), or in other words almost surely (a.s.),   (with respect to the reference probability measure $P$) $\w\in \O$, and  write
$Z_{t,u}\prec Z'_{t,u}$ to denote that $Z_{t,u}\preceq  Z'_{t,u}$
and $Z_{t,u}\ne Z'_{t,u}$.
By $\bbe_{|\F}$ or $\bbe[\,\cdot\,|\F]$ we denote  the conditional expectation with respect to sigma field $\F$.
 By $\one_A$ we denote the indicator function of set $A$.

Recall that a \emph{stopping time}, adapted  to the filtration $\cF$, is a random variable
$\tau\colon\O\to \{0,\dots,T\}$ such that $\{\w\in\O\colon\tau(\w)= t\}\in \F_t$ for $t=0,\dots,T$.
We denote by $\cT$
the set of   stopping times (adapted  to the filtration $\cF$).
For a process $Z_t\in \Z_t$, $t=0,\dots,T$, the (risk neutral) optimal stopping time problem can be written as
\begin{equation}\label{stop-1}
\minmax_{\tau\in \cT}\bbe[Z_\tau].
\end{equation}
By `$\minmax$' we mean that either  the minimization or maximization procedure  is applied.
Of course, the minimization procedure can be changed to the corresponding maximization simply by changing $Z_\tau$ to $-Z_\tau$.  Later we will consider combinations of stoping times and policy optimizations where it will be essential whether a maximization or minimization procedure is considered.

We will extend the above stopping time problem  in two directions.
First,  we combine this with cost optimization.
Consider objective  functions
$f_0\colon\bbr^{n_0}\to \bbr$,
$f_t\colon\bbr^{n_t}\times \O\to \bbr$,
and feasibility constraints defined by $\X_0\subset \bbr^{n_0}$ and  multifunctions
$\X_t\colon\bbr^{n_{t-1}}\times \O \rr\bbr^{n_t}$,  $t=1,\dots,T$.
It is assumed that $f_t(x_{t-1},\cdot)$ and $\X_t(x_{t-1},\cdot)$
are $\F_t$-measurable.  A sequence $\pi=\{x_0,\bx_1(\cdot),\dots,\bx_T(\cdot)\}$ of mappings $\bx_t\colon\O\to \bbr^{n_t}$, $t=0,\dots,T$,  adapted to the filtration\footnote{We use bold notation $\bx_t$ for (measurable) mappings in order to distinguish it from deterministic vector $x_t\in \bbr^{n_t}$.
Also by   writing $\bx_t(\cdot)$ we emphasize that this is a function of $\w\in \O$, i.e., is a random variable, rather than a deterministic vector. It is said that the sequence $(x_0,\bx_1,\dots,\bx_T)$ is adapted to the filtration if $\bx_t(\cdot)$  is $\F_t$-measurable for every $t=1,\dots,T$.}  $\cF$ is called a \emph{policy} or a \emph{decision rule}. Since $\F_0$ is trivial, the first decision $x_0$ is deterministic.
A policy $\pi=\{x_0,\bx_1,\dots,\bx_{T}\}$ is feasible if it satisfies the feasibility constraints, i.e.,  $x_0\in \X_0$ and
$\bx_t(\w)\in  \X_t(\bx_{t-1}(\w),\w)$, $t=1,\dots,T$,  for a.e. $\w\in \O$. We denote by $\Pi$ the set of feasible policies such that  $\big(f_0(x_{0}), f_1(\bx_1(\cdot),\cdot), \dots , f_T(\bx_{T}(\cdot),\cdot)\big)\in \Z_{0,T}$.
We then consider the following   problems
\begin{equation}
\label{stop-2}
\minmax_{\pi\in \Pi}\,\minmax_{\tau\in \cT} \bbe[f_\tau(\bx_\tau,\w)].
\end{equation}
The min-min, i.e., minimization with respect to $\pi\in \Pi$  and $\tau\in \cT$,   problem~\eqref{stop-2} is a natural formulation  aimed at the simultaneous minimization   with respect to the   decision rules  and stopping time. Similarly the max-max formulation is  aimed at the simultaneous maximization of $\pi\in \Pi$  and $\tau\in \cT$.
The max-min (min-max) formulation~\eqref{stop-2} could be interpreted as a certain type of compromise, we will discuss this later.

Second, we consider risk averse counterparts  of risk neutral problems of the form~\eqref{stop-1}--\eqref{stop-2}. It is tempting to extend the risk neutral formulations  simply by   replacing the corresponding expectation operator by an appropriate risk measure (such approach was adopted in some recent publications, e.g.,  \citet[Section 6.5]{Follmer2004},  \cite{Belom2017}, \cite{Goldenshluger}).  However, this has a delicate issue with   time consistency considerations.
We will discuss time consistency concepts from a somewhat general point of view in the following Sections~\ref{sec-timecon} and~\ref{sec:decomp} and will come back to discussion of optimal stopping time in Section~\ref{sec:Stopping}.

Let us finish this section with the following observations.
We have that $\O$ is the union of the disjoint sets
\begin{equation}\label{setomega}
\O_t^\tau:=\{\w\colon\tau(\w)=t\}, \;t=0,\dots,T,
\end{equation}
  and hence
$\one_\O=\sum_{t=0}^T\one_{\{\tau=t\}}$.
Note that $\one_{\{\tau=s\}}$ are $\F_t$ measurable for $0\le s\le t\le T$  and hence $\one_{\{\tau\le t\}}$ and $\one_{\{\tau> t\}}=\one_\O-\one_{\{\tau\le t\}}$ are also $\F_t$ measurable for  $t=0,\dots,T$.
Moreover
$\one_{\{\tau=t\}} Z_\tau=\one_{\{\tau=t\}} Z_t$
and thus for $(Z_0,\dots,Z_T)\in \Z_{0,T}$  it follows that
\begin{equation}\label{stopp-1}
 \begin{array}{lll}
 Z_\tau = \sum_{t=0}^T\one_{\{\tau=t\}}Z_\tau =
 \sum_{t=0}^T\one_{\{\tau=t\}}Z_t.
  \end{array}
\end{equation}
By replacing $Z_t$ in~\eqref{stopp-1} with $Z'_t:=Z_0+\dots+Z_t$ we obtain
\begin{align} \nonumber
	Z_0+\dots+Z_\tau  &=   Z'_\tau = \sum_{t=0}^T\one_{\{\tau=t\}}Z'_t
	 =  \sum_{t=0}^T\sum_{i=0}^t\one_{\{\tau=t\}}Z_i  \\
\label{stopp-1a}
	&= \sum_{t=0}^T\sum_{i=t}^T\one_{\{\tau=i\}}Z_t
	 = \sum_{t=0}^T\one_{\{\tau\ge t\}}Z_t.
\end{align}
The problem $\max_{\tau\in \cT}\bbe\left[\sum_{i=1}^\tau Z_i\right]$
can be considered in the framework of problem~\eqref{stop-1} simply by replacing $Z_\tau$ with $Z'_\tau=Z_0+\dots+Z_\tau$, and similarly for problems~\eqref{stop-2}.
Finally
note that since  $\F_0=\{\emptyset,\,\O\}$ is trivial, we have that
$\{\w\colon\tau(\w)=0\}$ is either  $\O$ or $\emptyset$, and hence either
$\one_{\{\tau=0\}}Z_0=Z_0$ or $\one_{\{\tau=0\}}Z_0=0$, and probability of the event $\{\tau=0\}$ is either $0$ or $1$.

\medskip
Because random variables are defined only up to $P$-almost sure equivalence,
it is in general not meaningful to speak of an `$\w$ by $\w$'
infimum
$\inf\{X(\w)\colon X \in\X\}$.  The essential infimum substitutes for this concept. 
We use the following concept of  essential infimum, cf.\  \citet[Appendix~A]{Karatzas}.
\begin{definition}[Essential infimum]
\label{def-ess}
 Let  $\X$ be a
nonempty family of   random variables defined on $(\O,\F,P)$.
The essential infimum of $\X$, denoted by $ \essinf  \X$, is a random variable $X^*$ satisfying: (i) if $X\in \X$, then $X\ge X^*$ a.s., and (ii) if $Y $ is a random variable satisfying $X \ge Y$ a.s. for all $X\in \X$, then $X^*\ge Y$ a.s.  The essential supremum is defined in the similar way.
\end{definition}

\section{Time consistency}
\label{sec-timecon}

In this section we discuss time consistency of multistage optimization problems from  a general point of view.
Consider a functional  $\R\colon\Z_{0,T}\to \bbr$
and  the optimization problem
\begin{equation}\label{refm-2}
	\min_{\pi\in \Pi}\R\big [f_0(x_0),\, f_1(
	\bx_1(\cdot),\cdot),\, \dots,\, f_T(\bx_T(\cdot),\cdot)
	\big],
\end{equation}
where $\Pi$ is the set of feasible policies as defined in  Section \ref{sec-optst}.
We refer to $\R$ as a \emph{preference functional} and to~\eqref{refm-2} as the \emph{reference} problem.

The principle of optimality, going back  to  \citet{Bellman},  postulates that an optimal policy computed at the initial stage of the decision process, before any realization of the uncertainty data became available, remains optimal at the later stages. This formulation is quite vague since it is not clearly specified what optimality at the later stages   does mean. In some situations this comes naturally and implicitly assumed. However, in more complex cases this could lead to a confusion and misunderstandings. Therefore, in order to proceed, we consider a class of
preference relations between possible realizations of random data   defined by a family  of  mappings
\begin{equation*}
	\R_{t,u}\colon\Z_{t,u}\to \Z_t,\quad 0\le t<u\le T.
\end{equation*}
We refer to each $\R_{t,u}$ as a \emph{preference mapping} and to the family
$\cR=\{\R_{t,u}\}_{0\le t<u\le T}$ as a \emph{preference system}.
Since $\Z_0$ is identified with  $\bbr$, we view   $\R_{0,T}(Z_0,\dots,Z_T)$   as  a real number for any $(Z_0,\dots,Z_T)\in\Z_{0,T}$.
We assume that $\R_{0,T}$ coincides with the preference functional $\R$ of the reference problem~\eqref{refm-2}, i.e., $\R=\R_{0,T}$.

\begin{remark}
\label{rem-risk}
The considered preference system does not necessarily represents a risk averse approach.
As we will discuss later, indeed in some settings it can be interpreted as conservative, and hence as  risk averse, while in other cases it can be viewed as somewhat optimistic (see, in particular,  Remark \ref{rem-max} below).
\end{remark}

\begin{definition}[Time consistent policies]\label{def-1a}
	We say that an optimal policy
	$\bar{\pi}=\{\bar{x}_0,\dots,\bar{\bx}_T\}$,
	solving the reference  problem~\eqref{refm-2},  is \emph{time consistent} with respect to the preference system~$\cR$, if at stages  $t=1,\dots,T-1$,  the policy $\{\bar{\bx}_t,\dots,\bar{\bx}_T\}$ is optimal for
	\begin{equation}
		\label{refm-3}
		\begin{array}{cll}
		\essinf&\R_{t,T}\big [f_t(
		\bx_t(\cdot),\cdot) ,  \dots , f_T(\bx_T(\cdot),\cdot)
		\big],\\
		{\rm s.t.} &
		 \bx_u(\cdot)\in \X_u(\bx_{u-1}(\cdot),\cdot) \;\text{a.s.},\;		u=t,\dots,T,
		\end{array}
	\end{equation}
 given $\bx_{t-1}=\bar{\bx}_{t-1}$.
\end{definition}

\begin{remark}
 At stage $t=1,\dots,T$, we already observed a realization of
 $\bar{\bx}_{t-1}$ of the policy~$\bar{\pi}$. The next decision $\bx_t$ should satisfy the feasibility constraint
 $ \bx_t(\cdot)\in \X_t(\bar{\bx}_{t-1}(\cdot),\cdot)$, and so on at the later stages. In that sense the optimization in  \eqref{refm-3} is performed over   policies $(\bx_t,\dots,\bx_T)$, satisfying the respective feasibility constraints, given $\bar{\bx}_{t-1}$.
 For definition of the concept of essential infimum,  used in \eqref{refm-3}, see Definition \ref{def-ess}.
 \end{remark}

It could be noted that we allow  for the preference mapping $\R_{t,T}$, $t=1,\dots,T-1$, to depend on realizations of the data process up to time $t$, i.e., we have that $\R_{t,T}(Z_t,\dots,Z_T)$ is $\F_t$-measurable. However, we do not allow $\R_{t,T}$  to depend on the decisions.
Definition~\ref{def-1a}  formalizes the meaning of  optimality of a solution of the reference problem  at the later stages of the decision process. Clearly this framework  depends on a choice of the preference system
 $\cR=\{\R_{t,u}\}_{1\le t<u\le T}$. This suggests the following  basic questions: (i)~what would be a `natural' choice of  mappings  $\R_{t,u}$, (ii)~what properties of mappings $\R_{t,u}$ are sufficient/\,necessary to ensure that every (at least one)  optimal solution of the reference problem is time consistent, (iii)~how time consistency is related to dynamic programming equations. As we shall see the last question is closely related to decomposability of $\R_{t,T}$ in terms of one-step mappings $\R_{t,t+1}$.

The minimal property that is required for the preference mappings is monotonicity.\footnote{In some publications the  monotonicity property  is understood in the reverse sense, i.e., $Z_{t,u}\preceq Z^\prime_{t,u}$  implies that
$\R(Z_{t,u})\se \R(Z^\prime_{t,u})$. In this paper we consider the  monotonicity only in the sense of Definition~\ref{def-mon}.}

\begin{definition}[Monotonicity]\label{def-mon}
We say that preference mapping $\R_{t,u}$ is \emph{monotone} if for any $Z_{t,u},Z'_{t,u}\in \Z_{t,u}$ such that  $Z_{t,u}\preceq  Z'_{t,u}$ it follows that
$\R(Z_{t,u})\preceq  \R(Z'_{t,u})$. We say that preference mapping $\R_{t,u}$ is \emph{strictly  monotone} if for any $Z_{t,u},Z'_{t,u}\in \Z_{t,u}$ such that  $Z_{t,u}\prec  Z'_{t,u}$ it follows that
$\R(Z_{t,u})\prec \R(Z'_{t,u})$. The preference system $\cR=\{\R_{t,u}\}_{0\le t<u\le T}$ is said to be monotone (strictly monotone) if every preference mapping $\R_{t,u}$ is monotone (strictly monotone).
\end{definition}

 \subsection{Additive case}
 \label{rem-addit}


The above framework~\eqref{refm-2}--\eqref{refm-3}  is very general. The   common case considered in the recent literature on risk averse stochastic optimization  is when the reference problem is a function of the total cost.
That is, when each $\R_{t,u}$ is  representable  as a function of $Z_t+\dots+Z_u$, i.e.,
	\begin{equation}\label{pref-add}
		\R_{t,u}(Z_t,\dots,Z_u):=\rho_{t,T}(Z_t+\dots+Z_u),\ 0\le t<u\le T,
	\end{equation}
	for some\footnote{Note that since $\Z_u\subset \Z_T$, $u=t+1,\dots,T$,  the corresponding mapping $\rho_{t,T}\colon\Z_u\to \Z_t$ is defined as restriction of $\rho_{t,T}$ to $\Z_u$.}
	$\rho_{t,T}\colon\Z_T\to \Z_t$, $t=0,\dots,T$. We will refer to such framework as the \emph{additive case}.  We will also consider  examples of natural and important
preference  systems which cannot be considered in the additive  framework~\eqref{pref-add}.

Consider the following properties of the mapping $\rho_{t,T}$:
\begin{enumerate}[noitemsep] 
	\item $\rho_{t,T}(\lambda Z+(1-\lambda)Z')\preceq \lambda\rho_{t,T}(Z)  +
	(1-\lambda)\rho_{t,T}(Z'),\;\;Z,Z'\in \Z_T,\,\;\lambda\in [0,1]$,
	\item $\rho_{t,T}(Z+ Z')\preceq \rho_{t,T}(Z)  +
	\rho_{t,T}(Z'),\;\;Z,Z'\in \Z_T$,
	\item $\rho_{t,T}(Z+ Z')\se \rho_{t,T}(Z)  +
	\rho_{t,T}(Z'),\;\;Z,Z'\in \Z_T$,
	\item $\rho_{t,T}(\lambda Z)=\lambda  \rho_{t,T}(Z),\;\; Z\in \Z_T,\;\lambda\ge 0$ and
	\item $\rho_{t,T}(Z +Z_t)=\rho_{t,T}(Z)+Z_t,\;\;   Z \in \Z_T, Z_t\in \Z_t$.
\end{enumerate}
We refer to these properties as \emph{convexity, subadditivity, superadditivity,  positive homogeneity and  translation equivariance,} respectively.
Note that  if  $\rho_{t,T}$ is positively homogeneous, then $\rho_{t,T}$  is convex iff it is subadditive.

With convex mapping $\rho_{t,T}$ is associated its concave   counterpart
\begin{equation}\label{concave}
\nu_{t,T}(Z):=-\rho_{t,T}(-Z).
\end{equation}
Note that $\nu_{t,T}(Z)$ inherits monotonicity, positive homogeneity and  translation equivariance properties of $\rho_{t,T}$,  and subadditivity of  $\rho_{t,T}$  implies superadditivity of $\nu_{t,T}$.
Following \cite{artzner1999} we refer to  $\rho_{t,T}$ as (convex)  \emph{coherent} if it is subadditive, monotone, positively homogeneous and   translation equivariant.

The preference  system~\eqref{pref-add} is monotone (strictly monotone) if $\rho_{t,T}$  are monotone (strictly monotone).
In particular, let $\varrho\colon L_p(\O,\F,P)\to\bbr$ be a law invariant coherent risk measure and $\varrho_{|\F_t}$ be its conditional analogue. Then $\rho_{t,T}:=\varrho_{|\F_t}$ is the corresponding coherent mapping.
	When $\varrho:=\bbe$ is the expectation operator, the corresponding  preference system is given by conditional expectations
	\begin{equation*}
		\R_{t,u}(Z_t,\dots,Z_u) = \bbe_{|\F_t} [Z_t+\dots+Z_u],\ 0\le t<u\le T,
	\end{equation*}
	which corresponds to the risk neutral setting.
	As another example take $\varrho:=\avr_\alpha$, where the Average Value-at-Risk measure can be defined as
	\begin{equation}\label{avr-2}
		\avr_\alpha (Z):=\inf_{u\in \bbr}\left\{u+(1-\alpha)^{-1}\bbe[Z-u]_+\right\},\; \alpha\in [0,1).
	\end{equation}
	For $\alpha=0$ the $\avr_0$ is the expectation operator and for $\alpha=1$ it becomes  $\varrho =\ess$, the essential supremum operator. The risk measure $\avr_\alpha$ is monotone but is not strictly monotone for $\alpha\ne 0$.
	The conditional analogue of the Average Value-at-Risk~\eqref{avr-2} is
	\begin{equation*}
		\avr_{\alpha_t\mid\mathcal F_t}(Z)=\essinf_{u_t\in L_\infty(\mathcal F_t)}\left\{u_t+(1-\alpha_t)^{-1}\bbe_{\mid\mathcal F_t}[Z-u_t]_+\right\},\; \alpha_t\in [0,1).
	\end{equation*}

\begin{example}
\label{ex-2}
Let $\Z_t:=L_\infty(\O,\F_t,P)$ and\footnote{The conditional essential supremum  is the smallest $\F_t$-random variable ($X_t$, say), so that $Z_{\tau} \preceq X_t$ for all $\tau$ with $t\le \tau\le u$ (see Definition \ref{def-ess}).
}
	\begin{equation*}
		\R_{t,u}(Z_t,\dots,Z_u):=\esssup_{\mathcal F_t} \{Z_{t},\dots,Z_u\},\ 0\le t<u\le T.
	\end{equation*}
	  The objective of the  corresponding reference problem~\eqref{refm-2} is then given by the maximum of the essential supremum of the   cost functions in the periods $t=0,\dots,T$.
\end{example}

 In the additive case,      the value  $\R_{t,u}(Z_{t},\dots,Z_u)$ is a function of the sum $Z_{t}+\dots+Z_u$ and is the same as value of  $\R_{t,T}$ applied to $Z_t+\dots+Z_u$ for any $u=t+1,\dots,T$. That is, in that framework there is no point of considering preference mappings $\R_{t,u}$ for $u$ different from  $T$. On the other hand, the preference system of Example~\ref{ex-2} is not additive and it is essential there to consider preference mappings $\R_{t,u}$ for $u<T$.

 \subsection{Two stage setting}
\label{sec:twost}

It is informative at this point to discuss the two stage case, $T=2$, since time inconsistency could  already happen there.
Consider   the following two stage   stochastic program
\begin{equation}\label{two-1}
	\begin{array}{cll}
		\min\limits_{x_0,\,\bxx_1\in \cX}   & \R\big(f_0(x_0),f_1(\bx_1(\cdot),\w)\big),\\
		 {\rm s.t. } & x_0\in \X_0,\;\bx_1(\w)\in \X_1(x_0,\w),
	\end{array}
\end{equation}
where
$\cX :=\left\{\bx_1\colon\O\to\bbr^{n_1}\mid f_1(\bx_1(\cdot),\cdot)\in \Z
\right\}$ with $\Z$ being a linear space of measurable functions $Z\colon\O\to \bbr$.
 The preference
functional  $\R\colon\Z_{0,1}\to \bbr$ is defined on the space $\Z_{0,1}=\bbr\times \Z$.
In order to deal with duality issues we consider the following  two frameworks for defining the space $\Z$. In one framework
we use, as in the previous section, $\Z:=L_p(\O,\F,P)$, $p\in [1,\infty]$, with $P$ viewed as the reference probability measure (distribution).
The space $\Z$ is paired with  the space $\Z^*=L_q(\O,\F,P)$, where $q\in [1,\infty]$ is such that $1/p+1/q=1$, with the respective bilinear form
\begin{equation*}
  \lan\zeta,Z\ran=\int_\O \zeta(\w)Z(\w)P(d\w),\;\zeta\in \Z^*,\;Z\in \Z.
\end{equation*}

This framework became standard in the theory of coherent risk measures. However, it is not applicable to the distributionally robust approach when the set of ambiguous distributions is defined by moment constraints and there is no reference probability measure.
In that case we assume that $\O$ is a compact metric space and
use the space
 $\Z:=C(\O)$ of continuous functions $Z\colon\O\to \bbr$. The dual $\Z^*$  of this space is
the space of finite signed Borel measures with the respective bilinear form
\begin{equation*}
  \lan\mu,Z\ran=\int_\O  Z(\w) \mu(d\w),\;\mu\in \Z^*,\;Z\in \Z.
\end{equation*}
Note that in the first framework of $\Z=L_p(\O,\F,P)$, an element $Z\in\Z$ actually
is a class of functions which can differ from each other on a set of $P$-measure zero.

The optimization in~\eqref{two-1} is performed over $x_0\in \X_0\subset \bbr^{n_0}$ and $\bx_1\colon\O\to\bbr^{n_1}$ such that $f_1(\bx_1(\cdot),\cdot)\in \Z$. The feasibility constraints $\bx_1(\w)\in \X_1(x_0,\w)$ in problem~\eqref{two-1} should be satisfied for  almost every    $\w\in \O$ in the first framework, and for \emph{all} $\w\in \O$ in the second framework of $\Z=C(\O)$. For $Z,Z'\in \Z$ we use the notation $Z\preceq Z^\prime$ to denote that $Z(\w)\le Z^\prime(\w)$ for a.e.\ $\w\in \O$ in the first framework, and for all $\w\in \O$ in the second framework.
As before we write $Z\prec Z'$ to denote that $Z\preceq Z'$ and $Z\ne Z'$.

The common setting considered in the stochastic programming literature is to define $\R(Z_0,Z_1):=\rho(Z_0+Z_1)$, $(Z_0,Z_1)\in \bbr\times\Z$,  where $\rho\colon\Z\to\bbr$ is a specified risk functional. This is the additive case discussed in Section \ref{rem-addit}.  In particular, when $\rho$ is the expectation operator this becomes the risk neutral formulation. However, there are many other possibilities to define preference
functionals $\R$ which are useful in various situations. For example, consider\footnote{We use notation $a\vee b=\max\{a,b\}$ and $a\wedge b=\min\{a,b\}$ for $a,b\in \bbr$.}
$\R(Z_0,Z_1):= Z_0\vee \rho(Z_1)$. If moreover, in the framework of $\Z=C(\O)$, we take $\rho(Z):=\sup_{\w\in \O}Z(\w)$, then the corresponding problem~\eqref{two-1} can be viewed as a robust type problem with minimization of the worst possible outcome of the two stages. That is, if the second stage cost is bigger than the first stage cost for some scenarios, then   the worst second stage cost is minimized. On the other hand, if the first stage cost is bigger for all scenarios, then the second stage problem is not considered.
Similarly in the framework of $\Z=L_\infty(\O,\F,P)$, we can take $\rho(Z):=\ess (Z)$. This  is the case of Example~\ref{ex-2}. As we shall discuss later  this is closely related to the problem of optimal stopping time.

\bigskip
In order to proceed we will need the following interchangeability result for a  functional
$\varrho\colon\Z\to\bbr$. Consider a function $\psi\colon\bbr^n\times \O\to \bbr\cup\{+\infty\}$. Let
\[\Psi(\w):=\inf_{y\in \bbr^n}\psi(y,\w)\] and
\[\cY:=\left\{\eta\colon\O\to\bbr^n\,  |\, \psi_\eta(\cdot)\in\Z \right\},
\]
where $\psi_\eta(\cdot):=\psi\big(\eta(\cdot),\cdot\big)$. 
\begin{assum}
\label{ass-1}
In the framework of $\Z=L_p(\O,\F,P)$,
suppose that the function  $\psi(y,\w)$ is random lower semicontinous, i.e., its epigraphical mapping is closed valued and measurable (such functions are called   normal integrands  in \cite{WetsRockafellar97}). In the framework of $\Z=C(\O)$ suppose that the minimum of $\psi(y,\w)$ over $y\in \bbr^n$  is attained for all $\w\in \O$.
\end{assum}
We have the following result about interchangeability of the minimum and preference
functionals (cf.\ \cite{shapiro2017}).

\begin{proposition}[Interchangeability principle]\label{pr-inter}
Suppose that  Assumption~\ref{ass-1} holds,  $\Psi\in \Z$ and $\varrho\colon\Z\to\bbr$ is monotone. Then
\begin{equation}\label{intpr-1}
	\varrho(\Psi)=\inf_{\eta\in \cY}\varrho(\psi_\eta)
\end{equation}
and
\begin{equation}\label{intpr-2}
	\bar{\eta}(\cdot)\in\argmin_{y\in \bbr^n} \psi(y,\cdot) \text{ implies }\bar{\eta}\in \argmin_{\eta\in \cY}\varrho(\psi_\eta).
\end{equation}
If moreover $\varrho$  is strictly monotone, then the converse of~\eqref{intpr-2} holds true as well, i.e.,
	\begin{equation}\label{intpr-3}
		\bar{\eta}\in \argmin_{\eta\in \cY}\varrho(\psi_\eta) \text{ implies }\bar{\eta}(\cdot)\in\argmin_{y\in \bbr^n} \psi(y,\cdot).
	\end{equation}
\end{proposition}

In the framework of $\Z=L_p(\O,\F,P)$,  it is assumed that $\psi(y,\w)$ is random lower semicontinous. It follows that the optimal value function $\Psi(\cdot)$ and the  multifunction $\cG(\cdot):=\argmin_{y\in \bbr^n} \psi(y,\cdot)$ are measurable \cite[Chapter 14]{WetsRockafellar97}. In that framework, the
meaning of the left hand side of~\eqref{intpr-2} and   right hand side of~\eqref{intpr-3} is that $\bar{\eta}(\cdot)$   is a measurable selection of $\cG(\cdot)$.

Equation~\eqref{intpr-1} means that the minimization and preference
functionals can be interchanged, provided that the preference
functional is monotone.
Moreover, the pointwise  minimizer in the left  hand side of~\eqref{intpr-2}, if it exists, solves the corresponding  minimization problem in the right hand side.  In order to conclude the inverse implication~\eqref{intpr-3}, that the corresponding  optimal functional solution is also the pointwise minimizer, the stronger condition of strict monotonicity is needed.

Consider now the problem~\eqref{two-1}, that   depends on $\w$,  and let
\begin{equation*}\label{secst-1}
	V(x_0,\cdot):=\inf_{x_1\in \X_1(x_0,\cdot)}\R\big(f_0(x_0),f_1(x_1,\cdot)\big),
\end{equation*}
which can be viewed   as value of the second stage problem.
By Proposition~\ref{pr-inter} we have the following.

\begin{theorem}\label{th-twost}
Suppose that:
{\rm (i)} the functional $\varrho(\cdot):=\R(Z_0,\cdot)$ is monotone for any $Z_0\in \bbr$,
	{\rm (ii)}  $V(x_0,\cdot)\in \Z$ for all $Z_0\in \bbr$,  and
{\rm (iii)} Assumption~\ref{ass-1} holds for \[\psi(x_1,\w):=\begin{cases}
	f_2(x_1,\w) &{\rm if }\; x_1\in \X_1(x_0,\w),\\
	+\infty & {\rm if }\; x_1\not\in \X_1(x_0,\w).
	\end{cases}\]

Then   the optimal value of the problem~\eqref{two-1} is the  same as the optimal value of
\begin{equation}\label{two-3}
	\Min_{x_0\in \X_0} \R\big(f_1(x_0),V(x_0,\cdot)\big).
\end{equation}
Further, $(\bar{x}_0,\bar{\bx}_1)$ is an optimal solution of the two stage problem~\eqref{two-1} if $\bar{x}_0$ is an optimal solution of the first stage  problem~\eqref{two-3} and\,\footnote{By writing `$\bar{\bx}_1(\cdot)\in\dots$' we mean that  $\bar{\bx}_1(\cdot)$ is a measurable selection and  such inclusion holds for a.e.\ $\w\in \O$ in the setting of $\Z=L_p(\O,\F,P)$, and for all $\w\in \O$ in the setting of $\Z=C(\O)$.}
\begin{equation}\label{secst-3}
	\bar{\bx}_1(\cdot)\in\argmin_{x_1\in \X_1(\bar{x}_0,\cdot)}\R\big(f_0(\bar{x}_0),f_1(x_1(\cdot),\cdot)\big).
\end{equation}
 Moreover,  if $\R(Z_0,\cdot)$ is \emph{strictly} monotone, then $(\bar{x}_0,\bar{\bx}_1)$ is an optimal solution of  problem~\eqref{two-1} if and only if $\bar{x}_0$ is an optimal solution of the first stage  problem and~\eqref{secst-3} holds.
 \end{theorem}

Here, time consistency of a policy $(\bar{x}_0,\bar{\bx}_1)$, solving the two stage  problem~\eqref{two-1}, means that $\bar{\bx}_1$ solves the respective second stage problem given  $x_0=\bar{x}_0$, i.e., condition~\eqref{secst-3} holds. That is, if $(\bar{x}_0,\bar{\bx}_1)$ is not  time consistent, then there exists another feasible solution $\tilde{\bx}_1$ such that
$ f_2(\tilde{\bx}_1(\cdot),\cdot)\prec f_2(\bar{\bx}_1(\cdot),\cdot)$.
Without \emph{strict} monotonicity of $\R$ it could happen that problem~\eqref{two-1} has optimal solutions which do not satisfy condition~\eqref{secst-3} and hence are not time consistent. That is, condition~\eqref{secst-3} is a sufficient but without strict monotonicity is not necessary for optimality. Such examples can be found, e.g., in \cite{shapiro2017} and for robust optimization were given in \cite{DelageIancu}.
For example, for  $\R(Z_0,Z_1)=Z_0\vee \rho(Z_1)$
we have that if
$Z^\prime_1\prec Z_1$
are such that  $\rho(Z^\prime_1)<\rho(Z_1)<Z_0$, then $\R(Z_0,Z_1)= \R(Z_0,Z^\prime_1)$. That is, $\R(Z_0,\cdot)$ is not strictly monotone. It could happen then that a second stage decision does not satisfy~\eqref{secst-3} and is not time consistent.

\subsection*{Dual representation}
The space $\Z_{0,1}=\bbr\times \Z$ can be equipped, for example,   with the norm $\|(Z_0,Z_1)\|_{0,1}:=|Z_0|+\|Z_1\|$, where $\|\cdot\|$ is the respective norm of the space $\Z$, and can be  paired with the space $\Z^*_{0,1}=\bbr\times \Z^*$ with the bilinear form
\[
\lan (\zeta_0,\zeta_1),(Z_0,Z_1)\ran:=\zeta_0\, Z_0+\lan  \zeta_1, Z_1 \ran,\;
(\zeta_0,\zeta_1)\in \Z^*_{0,1},\;(Z_0,Z_1)\in \Z_{0,1}.
\]
Suppose that the functional $\R\colon\Z_{0,1}\to \bbr$ is \emph{convex} and \emph{monotone}. Then by the Klee-Nachbin-Namioka Theorem  the functional $\R$ is continuous in the norm topology of $\Z_{0,1}$ (cf.\ \citet[Proposition 3.1]{Ruszczynski2006}). Suppose further that $\R$ is \emph{positively homogeneous}, i.e., $\R(t\, Z_{0,1})=t\, \R(Z_{0,1})$ for any $t\ge 0$ and $Z_{0,1}\in \Z_{0,1}$. Then
by the Fenchel-Moreau Theorem,
$\R$ has the dual representation
\begin{equation}\label{dualtwo}
 \R(Z_0,Z_1)=\sup_{(\zeta_0,\zeta_1)\in \cA_{0,1}}\lan (\zeta_0,\zeta_1),(Z_0,Z_1)\ran
\end{equation}
for some convex, bounded and weakly$^*$ closed set \[\cA_{0,1}\subset \{(\zeta_0,\zeta_1)\in \Z^*_{0,1}\colon \zeta_0\ge 0,\ \zeta_1\se 0\}\] (cf.\ \citet[Theorem~2.2]{Ruszczynski2006}). The subdifferential of  $\R$ is then given by
\begin{equation}\label{dualsub}
	\partial \R(Z_0,Z_1)=\argmax_{(\zeta_0,\zeta_1)\in \cA_{0,1}}\lan (\zeta_0,\zeta_1),(Z_0,Z_1)\ran.
\end{equation}
In particular $\partial \R(0,0)= \cA_{0,1}$.

\begin{example}
\label{ex-gen}
Consider preference
functional of the form $\R(Z_0,Z_1):=\varphi(Z_0,\rho(Z_1))$, where $\rho\colon\Z\to \bbr$ is a coherent risk measure and $\varphi\colon\bbr\times \bbr\to \bbr$ is a   convex monotone positively homogeneous   function. It follows then  that  $\R(\cdot,\cdot)$ is convex  monotone and positively homogeneous. Let $\rho(Z)=\sup_{\zeta\in \cA}\lan\zeta,Z\ran$ be the dual representation of $\rho$, where $\cA$ is a convex bounded weakly$^*$ closed   subset of $\Z^*$.
Then $\partial\varphi(x_0,x_1)$ consists of vectors (subgradients) $(\gamma_0,\gamma_1)$ such that
\[
\varphi(y_0,y_1)-\varphi(x_0,x_1)\ge \gamma_0(y_0-x_0)+\gamma_1(y_1-x_1)
\]
for all $(y_0,y_1)\in \bbr^2$. Since $\varphi$ is monotone, it follows that $\partial\varphi(x_0,x_1)\subset \bbr_+^2$.
Consequently the representation~\eqref{dualtwo} holds with
\begin{equation*}\label{dualsub-2}
	\cA_{0,1}=\partial  \R(0,0)=\left \{(\zeta_0,\zeta_1)\in \bbr\times \Z^*\colon\zeta_1=\gamma_1 \zeta,\;\zeta\in \cA,\;(\zeta_0,\gamma_1)\in \partial\varphi(0,0)\right\}.
\end{equation*}
For example let  $\varphi(x_0,x_1):=x_0\vee x_1$,  and hence  $\R(Z_0,Z_1)=Z_0\vee \rho(Z_1)$.
Then  $\partial\varphi(0,0)=\{(t,1-t)\colon t\in [0,1]\}$ and
\[
\cA_{0,1}=\{(t,(1-t)\zeta)\colon t\in [0,1],\; \zeta\in \cA\}.
\]

\end{example}

\subsection{Dynamic consistency of preference systems}\label{sec:TC}
Many authors investigate time consistency by addressing special properties on the preference system itself. This section recalls these concepts and relates these properties to time consistency of optimal policies. The following concept of dynamic consistency (also called time consistency by some authors), applied to the preference systems rather than considered policies, in slightly different forms was used by several authors (cf., \cite{Kreps1978,Wang1999,Epstein2003,Riedel2004, cher2006, Artzner2007, Ruszczynski2010}).

\begin{definition}[Dynamical consistency]\label{def-dc}
	The preference system $\{\R_{t,u}\}_{1\le t<u\le T}$ is said to be \emph{dynamically consistent} if  for    $1\le s<t<u \le T$ and   $(Z_s,\dots,Z_u),(Z'_s,\dots,Z'_u)\in \Z_{t,u}$ such that $Z_\tau=Z'_\tau,\;\tau=s,\dots,t-1$, the following `forward' implication  holds:
	\begin{equation}\label{appr-3}
		\text{if }\R_{t,u}(Z_t,\dots,Z_u)\preceq \R_{t,u}(Z'_t,\dots,Z'_u) \text{ then }\R_{s,u}(Z_s,\dots,Z_u)\preceq \R_{s,u}(Z'_s,\dots,Z'_u).
	\end{equation}
\end{definition}

It turns out that the above `forward' property of dynamic consistency is not always sufficient to ensure that every optimal policy is time consistent. For that we need a stronger notion of dynamic consistency (cf.\ \citet[Section~6.8.5]{RuszczynskiShapiro2009-2}).

\begin{definition}[Strict dynamical consistency]\label{def-4}
	A dynamically consistent
	preference system $\{\R_{t,u}\}_{1\le t<u\le T}$ is said to be \emph{strictly dynamically consistent} if in addition to~\eqref{appr-3} the following  implication holds:
	\begin{equation}\label{appr-4}
		\text{if }\R_{t,u}(Z_t,\dots,Z_u)\prec \R_{t,u}(Z'_t,\dots,Z'_u) \text{ then }\R_{s,u}(Z_s,\dots,Z_u)\prec \R_{s,u}(Z'_s,\dots,Z'_u)
	\end{equation}
 for all  $1\le s<t<u \le T$.
\end{definition}

Note that it follows from~\eqref{appr-3} that
\begin{equation*}\label{appr-eqv}
	\R_{t,u}(Z_t,\dots,Z_u)= \R_{t,u}(Z'_t,\dots,Z'_u) \text{ implies }\R_{s,u}(Z_s,\dots,Z_u)= \R_{s,u}(Z'_s,\dots,Z'_u).
\end{equation*}
Recall that in the additive case, $\R_{t,T}(Z_t,\dots,Z_T)$ is given by $\rho_{t,T}(Z_t+\dots+Z_T)$. In that case condition~\eqref{appr-3} implies that 
 \begin{equation}\label{mul-a}
			\text{if }Z,Z'\in \Z_T\; {\rm and } \;\rho_{t,T}(Z)\preceq \rho_{t,T}(Z') \text{ then }\rho_{s,T}(Z)\preceq \rho_{s,T}(Z'),\;\;1\le s<t\le T-1.
	\end{equation}
Conversely, if  moreover  $\rho_{s,T}$ is translation equivariant, then  we can write for $1\le s<t\le T$,
\[
\rho_{t,T}(Z_s+\dots+Z_T)=Z_s+\dots+Z_{t-1}+
\rho_{t,T}(Z_t+\dots+Z_T),
\]
and hence  condition~\eqref{mul-a} implies~\eqref{appr-3} for $u=T$.  If
$\rho_{t,T}:=\bbe_{|\F_t}$, then for $s<t$  we have that $\rho_{s,T}(Z)=\bbe_{|\F_s}[\,\bbe_{|\F_t}(Z)]$ and hence this preference system is dynamically consistent. In fact it is not difficult to see that this preference system is strictly dynamically consistent.

Similar to the additive case we have the following result (cf.\ \citet[Propositions~6.80]{RuszczynskiShapiro2009-2}).

\begin{proposition}\label{pr-din1}
	The following holds true:
	{\rm (i)} If the preference system is dynamically consistent and $\bar{\pi}\in \Pi$ is the unique optimal solution	of the reference problem~\eqref{refm-2}, then  $\bar{\pi}$ is time consistent.
{\rm (ii)} If the preference system is strictly  dynamically consistent, then  every optimal solution of the reference problem  is time consistent.
\end{proposition}

\section{Decomposability and dynamic equations}
\label{sec:decomp}

Let us start with definition of the following basic decomposability concept.
\begin{definition}[Recursivity]\label{def:rec-a}
	The preference system $\{\R_{t,u}\}_{0\le t<u\le T}$ is said to be \emph{recursive}, if
	\begin{equation}\label{eq:Rec}
		\R_{t,u} (Z_{t},\dots, Z_u )=
		\R_{t,v}\big(Z_t,\dots, Z_{v-1}, \R_{v,u}(Z_{v},\dots, Z_u)\big)
	\end{equation}
	for any $0\le t<v<u\le T$ and $(Z_t,\dots, Z_u)\in \Z_{t,u}$.
\end{definition}

We have the following relation between recursivity and dynamic consistency   (discussed in Section~\ref{sec:TC}).

\begin{proposition}
\label{pr-decom}
Suppose that  preference mappings $\R_{t,u}$, $1\le t<u\le T$, are monotone (strictly monotone) and  recursive. Then $\{\R_{t,u}\}_{1\le t<u\le T}$ is   dynamically consistent (strictly dynamically consistent).
\end{proposition}

\begin{proof}
We need to verify the implication~\eqref{appr-3}. By recursivity,
for $1\le s<t<u\le T$, we have
\begin{align*}
	\R_{s,u}(Z_s,\dots,Z_u) &=
	\R_{s,t}\big(Z_s,\dots, Z_{t-1}, \R_{t,u}(Z_{t},\dots, Z_u)\big)\text{ and}\\
	\R_{s,u}(Z'_s,\dots,Z'_u)&=
	\R_{s,t}\big(Z'_s,\dots, Z'_{t-1}, \R_{t,u}(Z'_{t},\dots, Z'_u)\big).
\end{align*}
Assuming $Z_\tau=Z'_\tau$, $\tau=s,\dots,t-1$ the implication~\eqref{appr-3} (the implication~\eqref{appr-4}) follows by monotonicity (strict monotonicity) of $\R_{s,t}$.
\end{proof}

It follows then by Proposition \ref{pr-din1}(ii) that if the system is recursive and the preference mappings are strictly monotone, then every optimal solution
of the reference problem is time consistent. As already mentioned, without \emph{strict} monotonicity the recursivity property does not necessarily imply  time consistency of \emph{every} optimal solution.
\\




In the additive case (discussed in Section~\ref{rem-addit}), when
$\R_{t,T}(Z_t,\dots,Z_T)=\rho_{t,T}(Z_{t}+\dots+Z_T)$, $t=0,\dots,T-1$, the recursive property can be written as
\begin{equation}\label{addrec}
	\rho_{t,T}(\rho_{v,T}(Z))=\rho_{t,T}(Z),\;Z\in \Z_T,\ 0\le t<v\le T-1.
\end{equation}
Note that since $\rho_{v,T}(Z)\in \Z_v$,  we have   that
$\rho_{t,T}(\rho_{v,T}(Z))=\rho_{t,v}(\rho_{v,T}(Z)).$
By applying~\eqref{addrec} recursively for $v=t+1,\dots,T-1$, this  means that $\rho_{t,T}$ can be decomposed as
\begin{equation}\label{sum-1} \rho_{t,T}(Z)= \rho_{t,T}\big(\rho_{t+1,T}\left (\cdots\rho_{T-1,T}(Z)\right )\big),\;Z\in \Z_T.
\end{equation}
If moreover $\rho_{t,T}$ is  translation equivariant, this becomes
\begin{equation}\label{sum-2}
	\rho_{t,T}(Z_t+\dots+Z_T)= Z_t+\rho_{t,T}\big(Z_{t+1}+
	\rho_{t+1,T}(Z_{t+2})+\cdots+
	\rho_{T-1,T}(Z_T)\big).
\end{equation}

For a law invariant convex coherent  measure $\varrho$ and $\rho_{t,T}:=\varrho_{|\F_t}$,  the  recursive property~\eqref{addrec} can hold only in two cases -- for the `expectation' and the $`\ess$'  operators (cf.\ \cite{Schachermayer2009a}). For example the Average Value-at-Risk preference system, $\rho_{t,T}:=\avr_{\alpha|\F_t}$, is not recursive for $\alpha\in (0,1)$. Recursive preference system, in the additive case, can be constructed in the  nested form
\begin{equation}\label{sum-3}
\rho_{t,T}(Z):= \phi_t\big(\phi_{t+1}\left (\cdots\phi_{T-1}(Z)\right )\big),\;Z\in \Z_T,
\end{equation}
where $\phi_s\colon\Z_{s+1}\to \Z_s$, $s=1,\dots,T-1$,  are one-step mappings. For example, taking $\varrho:=\avr_\alpha$ this becomes
\begin{equation*}\label{eq:nAVaR2}
	\rho_{t,T}(\cdot) =\avr_{\alpha|\F_{t}}\big (\avr_{\alpha|\F_{t+1}}(\cdots \avr_{\alpha|\F_{T-1}}(\cdot))\big),
\end{equation*}
the so-called \emph{nested Average Value-at-Risk} mappings. As it was pointed out above, for $\alpha\in (0,1)$ these nested Average Value-at-Risk  mappings  are different from the $\avr_{\alpha|\F_{t}}$.

Consider now the general case of the preference system $\{\R_{t,u}\}_{1\le t<u\le T}$. The recursive property~\eqref{eq:Rec} implies that $\R_{t,u}$ can be decomposed
in terms of one step mappings $\R_{s,s+1}$, $s=t,\dots,u-1$,  as
\begin{equation}\label{decom-1}
	\R_{t,u} (Z_t,\dots, Z_u )=
	\R_{t,t+1}\Big(Z_t, \R_{t+1,t+2}\big(Z_{t+1},\cdots, \R_{u-1,u}(Z_{u-1}, Z_u)\big)\Big).
	\end{equation}
Conversely recursive preference mappings can be constructed in the form~\eqref{decom-1} by choosing one step
 mappings $\R_{s,s+1}\colon\Z_s\times \Z_{s+1}\to\Z_s$, $s=1,\dots,T-1$.

\subsection{Dynamic programming equations}

The recursive  property~\eqref{eq:Rec} and monotonicity of the preference system allow to write the following dynamic programming equations for the reference problem~\eqref{refm-2}, derivations are similar to the two stage case discussed in Section~\ref{sec:twost}.
At the terminal stage~$T$ the cost-to-go function is defined as
\begin{equation}\label{dynmar-1}
V_T(x_{T-1},\w):=\essinf_{x_T\in \X_T(x_{T-1},\w)} f_T(x_T,\w).
\end{equation}
Suppose that $x_0,\dots,\bx_{T-1}$ are given.
Since $\R=\R_{0,T}$ we have by the interchangeability principle (Proposition~\ref{pr-inter}) and recursivity that
\begin{align}\label{eq:17}
	\inf_{\bsx_T \in \X_T(\bsx_{T-1},\cdot)}&\R\big[f_0(x_0) ,  \dots , f_T(\bx_{T}(\cdot),\cdot)\big]\\
	&= \R\big[f_0(x_0),  \dots , f_{T-1}(\bx_{T-1},\w), \inf_{\bsx_T \in \X_T(\bsx_{T-1},\w)} f_T(\bx_{T},\w)\big]\nonumber\\
	& = \R_{0,T}\big[f_0(x_0), \dots, f_{T-2}(\bx_{T-2},\w), f_{T-1}(\bx_{T-1},\w),\, V_T(x_{T-1},\w)\big]\nonumber\\
	&=\R_{0,T}
	\big[f_0(x_0), \dots,f_{T-2}(\bx_{T-2},\w), \R_{T-1,T}[f_{T-1}(\bx_{T-1},\w),\, V_T(x_{T-1},\w)]\big],\nonumber
\end{align}
assuming that $V_T(\bx_{T-1},\cdot)\in \Z_T$.
Continuing this backward in time we obtain
at  stages $t=T-1,\dots,1$, the  cost-to-go functions
\begin{equation}\label{dynmar-2}
V_t(x_{t-1},\w):= \essinf_{x_t\in \X_t(x_{t-1},\w)}\R_{t,t+1}
\big(f_t(x_t,\w),\ V_{t+1}(x_t,\w)\big),
\end{equation}
representing the corresponding dynamic programming equations. Finally, at the first stage, the problem
\begin{equation*}
	\min_{x_0\in\X_0}\R_{0,1}\big(f_0(x_0),\ V_1(x_0,\cdot)\big)
\end{equation*}
should be solved. We conclude with~\eqref{eq:17} that
\begin{equation}\label{dynmar-3}
V_0:=\min_{\pi\in \Pi}\R\big [f_0(x_0),\, f_1(\bx_1(\cdot),\cdot),\, \dots,\, f_T(\bx_T(\cdot),\cdot) \big]
\end{equation} for the recursive preference system $\R$.

In a rudimentary form such approach to writing dynamic equations with relation to time consistency was outlined in \cite{Shapiro2009}.

\begin{definition}[Dynamic programming equations]
	We say that a policy $\pi=(x_0,\bx_1,\dots,\bx_T)$  satisfies the \emph{dynamic programming equations}~\eqref{dynmar-1}, \eqref{dynmar-2} and~\eqref{dynmar-3}  if
	\begin{align}
		\bx_T(\cdot)&\in \argmin\limits_{x_T\in \X_T(\bxx_{T-1},\,\cdot)} f_T(x_T,\cdot),\label{suff0}\\
		\bx_{t}(\cdot)&\in \argmin\limits_{x_t\in \X_t(x_{t-1},\cdot)}\R_{t,t+1}
		\big(f_t(x_t,\cdot),\,V_{t+1}(x_t,\cdot)\big),
		\;t=1,\dots,T-1,\label{suff}\\
		x_0&\in \argmin\limits_{x_0\in\X_0}
		\R_{0,1} (f_0(x_0),\,V_1(x_0,\cdot) ).\label{suffT}
	\end{align}
\end{definition}
If a policy satisfies the  dynamic programming equations, then it is optimal for the reference multistage problem~\eqref{refm-2} and is time consistent. Without strict monotonicity it could happen that a policy, which is optimal for the reference  problem~\eqref{refm-2}, does not satisfy the  dynamic programming equations and is not time consistent. As it was discussed in Section~\ref{sec:twost} this could happen even in the two stage case and a finite number of scenarios.

\begin{remark}
Consider  the additive case where  $\R_{t,t+1}(Z_t,Z_{t+1})=\rho_{t,T}(Z_t+Z_{t+1})$. If moreover  mappings  $\rho_{t}$ are translation equivariant, this becomes
\begin{equation}\label{ones-0}
 \R_{t,t+1}(Z_t,Z_{t+1})=Z_t+
	\rho_{t,T}(Z_{t+1}).
\end{equation}
	Suppose further that  mappings $\rho_{t,T}$  are decomposable  via a family of
	one-step coherent    mappings $\phi_t$, as in~\eqref{sum-3}. In that case    equations~\eqref{dynmar-1}--\eqref{dynmar-3} coincide with the respective equations of the additive case (cf.\ \cite{Ruszczynski}).
\end{remark}

\begin{example}
  \label{ex-d1}
Let us define one step mappings as
\begin{equation}\label{onestep-1}
 \R_{s,s+1}(Z_s,Z_{s+1}):= Z_s\vee \varrho_{|\F_s}(Z_{s+1}),
 \quad s=0,\dots,T-1,
\end{equation}
and the corresponding preference  mappings of the form~\eqref{decom-1}, where $\varrho$ is a law invariant coherent measure. In particular for $\varrho:=\ess$ we obtain the preference system of Example~\ref{ex-2}. Here  the dynamic equations~\eqref{dynmar-2} take the form
\begin{equation}\label{dyn-stop1}
	V_t(x_{t-1},\w)= \essinf_{x_t\in \X_t(x_{t-1},\w)}
\left\{	 f_t(x_t,\w)\vee \varrho_{|\F_t}\left(V_{t+1}(x_t,\w)\right)\right\},
\end{equation}
and
the reference problem   can be viewed as minimization of the worst possible outcome over the considered period of time measured in terms of the  measure $\varrho$.

As we shall see in Section \ref{sec-mos}, this example is closely related to the stopping time risk averse formulation of multistage programs.

\end{example}

\section{Time consistent optimal stopping}
\label{sec:Stopping}

In this section we discuss a combination of the optimal stopping time and time consistent formulations of multiperiod  preference  measures.
\begin{definition}
	Let $\varrho_{t|\F_t}\colon\Z_{t+1}\to\Z_{t}$, $t=0,\dots,T-1$,
	be monotone, translation equivariant mappings and consider  the corresponding mappings $\rho_{s,t}\colon\Z_t\to\Z_s$ represented in the nested form
	\begin{equation}\label{stopp-2}
		\rho_{s,t}(\cdot):=\varrho_{s|\F_s} \big(\varrho_{s+1|\F_{s+1}}\left (\cdots\varrho_{t-1|\F_{t-1}}(\cdot)\right )\big),\;0\le s<t\le T.
	\end{equation}
	The \emph{stopping risk measure} is
	\begin{equation}\label{stopp-3}
		\rho_{0,T}(Z_\tau)= \one_{\{\tau=0\}}Z_0 + \varrho_{0|\F_0}\left(\one_{\{\tau=1\}}Z_1    +\cdots +\varrho_{T-1|\F_{T-1}}(\one_{\{\tau=T\}}Z_T)  \right).
	\end{equation}
\end{definition}
The stopping  risk measure is well-defined by virtue of~\eqref{stopp-1} and the translation equivariance of the mappings $\varrho_{t|\F_{t}}$. Since $\F_0$ is trivial, the corresponding functional $\ \rho_{0,T}\colon\Z_T\to \bbr$ is real valued.

In   the risk neutral case when  $\varrho_{t|\F_t}:=\bbe_{|\F_t}$, we have that $\rho_{s,t}=\bbe_{|\F_s}$ for $0\le s<t\le T$,
in particular   $\rho_{0,T}= \bbe_{|\F_0}=\bbe$,  hence
\begin{equation}
\label{stopp-4}
 \bbe(Z_\tau)= \bbe\left[\sum_{t=0}^T\one_{\{\tau=t\}}Z_t\right]=
 \one_{\{\tau=0\}}Z_0 + \bbe_{|\F_0}\left(\one_{\{\tau=1\}}Z_1    +\cdots +\bbe_{|\F_{T-1}}(\one_{\{\tau=T\}}Z_T)  \right).
\end{equation}
The stopping time risk measure suggests the following counterpart of the risk neutral stopping time problem~\eqref{stop-1} (recall that by `$\minmax$' we mean that  either   the minimization or maximization procedure is applied):
\begin{equation}\label{stop-risk}
\minmax_{\tau\in \cT}\rho_{0,T}(Z_\tau).
\end{equation}
As it was argued in the previous sections, the above  formulation~\eqref{stop-risk} can be viewed as time consistent and is amenable to writing the dynamic programming equations.

\subsection{Distributionally robust approach}


	It is  possible to view the  stopping time formulation~\eqref{stop-risk} from the following distributionally robust point of view.  Consider  a \emph{convex}  coherent  functional  $\varrho\colon L_p(\O,\F,P)\to\bbr$. It
 can be represented in the dual form
	\begin{equation}\label{riskdual}
		\varrho (Z)=\sup_{Q\in \cM}\bbe_Q[Z],
	\end{equation}
	where $\cM$ is a set of probability measures absolutely continuous with respect to the reference probability measure $P$ and such that the densities  $dQ/dP$, $Q\in \cM$, form a bounded convex weakly$^*$ closed set $\cA\subset L_q(\O,\F,P)$ in the dual space
	$L_q(\O,\F,P)=L_p(\O,\F,P)^*$. For the concave counterpart $-\varrho(-Z)$ of $\varrho$ (see~\eqref{concave}),  the corresponding dual representation is obtained by replacing  `sup' in~\eqref{riskdual} with `inf', that is
\[
-\varrho(-Z)=\inf_{Q\in \cM}\bbe_Q[Z].
\]

Conversely, given a set $\cM$  of probability measures absolutely continuous with respect to the reference probability measure $P$ and such that the densities  $dQ/dP$, $Q\in \cM$, form a bounded convex weakly$^*$ closed set $\cA\subset L_q(\O,\F,P)$, one can use the righthand side of~\eqref{riskdual} as a definition of the corresponding  functional $\varrho$.
The so defined   functional $\varrho$  is    convex coherent. It is  law invariant iff the set $\cA$ is law invariant in the sense that if $\zeta\in \cA$ and $\zeta'$ is a density distributionally equivalent\footnote{It is said that $\zeta$ and $\zeta'$ are distributionally equivalent if $P(\zeta\le z)=P(\zeta'\le z)$ for all $z\in \bbr$.}   to $\zeta$, then $\zeta'\in \cA$. This holds even if the reference probability measure $P$ has atoms (cf.\ \citet[Theorem~2.3]{shapiro2017}).

Since $\bbe_Q[\,\cdot\,]= \bbe_Q\big[\bbe_Q [\,\cdot\,|\F_t]\big]$, for $Q\in \cM$,  we have that
	\begin{equation}\label{recsup}
	 \sup_{Q\in \cM}\bbe_Q[\,\cdot\,]=\sup_{Q\in \cM}\bbe_Q\big[\bbe_Q [\,\cdot\,|\F_t]\big ]\le
	\sup_{Q\in \cM}\bbe_Q\big[\,\esssup_{Q\in \cM}\bbe_Q [\,\cdot\,|\F_t]\big ].
	\end{equation}
The functional  $\varrho_{|\F_t}(\cdot):=\esssup_{Q\in \cM}\bbe_Q [\,\cdot\,|\F_t]$ can be viewed as the conditional counterpart  of the corresponding functional $\varrho$.
	Equality in~\eqref{recsup}  would mean  the recursive property  $\varrho(\cdot)= \varrho(\varrho_{|\F_t}(\cdot))$.  In the law invariant case and when the reference probability measure is nonatomic,
the   functional $\varrho$ has such  recursive property
  only when the set $\cM$ is a singleton or consists of all probability measures absolutely continuous with respect to the reference measure (cf.\ \cite{Schachermayer2009a}).
	
	The nested functional  $\rho_{0,T}$, defined in~\eqref{stopp-2}, is decomposable, i.e., has the recursive property~\eqref{addrec}.
	For   not decomposable  (law invariant)  risk measure $\varrho$ the corresponding \emph{nested} stopping objective
	$\rho_{0,T}(Z_\tau)$, of the form~\eqref{stopp-3},   is different from $\varrho(Z_\tau)$. As we shall see below the \emph{nested} formulation of   stopping time is amenable  for writing dynamic programming equations,  and in the  sense of nested decomposition  is time consistent.

 \begin{remark}
\label{rem-stab}
In some recent publications it was suggested to consider the following formulation of distributionally robust (risk averse) optimal stopping time problems
\begin{equation}\label{form}
  \max_{\tau\in \cT}\left\{\varrho(Z_\tau):=\inf_{Q\in \cM} \bbe_Q[Z_\tau]\right\}.
\end{equation}
As it is pointed above,
unless the set $\cM$ is a singleton or consists of all probability measures, this is not the same as the corresponding nested formulation. In order to deal with this, and eventually to write the associated dynamic equations, it was assumed in
\citet[Section 6.5]{Follmer2004} that the set $\cM$  possesses a certain property, called stability.  By the above discussion it appears that such stability property will hold only in rather exceptional cases.
 \end{remark}

 \begin{remark}
\label{rem-max}
Consider the maximization   variant of  problem~\eqref{stop-risk}.
By the dual representation~\eqref{riskdual},  such formulation  with \emph{convex} preference functional  can be viewed as somewhat optimistic;  hoping that the uncertainty of the probability distribution, represented by the respective set $\cM$,  works  in our favor  potentially giving a larger value of the return $Z_\tau$.
From the risk averse (pessimistic)  point of view it makes more sense either  to use    the corresponding  concave counterpart $\nu_{0,T}$
(defined in~\eqref{concave}), or to work with the minimization    variant while employing   the convex preference functional.
As we shall see later these considerations raise delicate issues of preserving convexity of considered problems which is crucial for efficient numerical procedures.
 \end{remark}

\subsection{Multistage risk averse optimal stopping time}
\label{sec-mos}
By employing nested functionals\footnote{Recall that the considered nested functionals are assumed to be
monotone and  translation equivariant, while can be convex or concave.}
$\rho_{0,T}$ of the form~\eqref{stopp-2},  and by using~\eqref{stopp-1},   the corresponding      optimal stopping time counterparts of problems~\eqref{stop-2}   can be written  as
\begin{equation}
\label{stp-max}
\minmax_{\pi\in \Pi}\,\minmax_{\tau\in \cT}\rho_{0,T}\left(f_\tau(\bx_\tau,\w)\right),
\end{equation}
 with
\begin{equation}
\label{refm-Stop2}
\begin{array}{lll}
\rho_{0,T}( f_\tau(\bx_{\tau},\w))  =&  \one_{\{\tau=0\}}  f_0(x_0) +
   \varrho_{0|\F_0}\big(\one_{\{\tau=1\}}f_1(\bx_{1},\w)+ \\
  &   \cdots
 +\varrho_{T-1|\F_{T-1}}(\one_{\{\tau=T\}}
f_T(\bx_{T},\w)  \big).
\end{array}
\end{equation}


\begin{remark}
It is also possible to consider optimization   of
$\rho_{0,T}\big(f_0(x_0) +\dots+ f_\tau(\bx_{\tau},\w)\big)$ by using~\eqref{stopp-1a} rather than~\eqref{stopp-1}, i.e., by replacing the cost function  $f_\tau(\bx_{\tau},\w)$ with the cumulative cost
$f_0(x_0) +\dots+ f_\tau(\bx_{\tau},\w)$
 (for the risk neutral case and fixed stopping time $\tau$  this type of problems  were  considered recently  in  \cite{Guigues2018}).
 These formulations are equivalent and we concentrate below on the formulations~\eqref{stp-max}.
\end{remark}

The problems~\eqref{stp-max}   are  twofold, they consist  in finding simultaneously
an optimal policy $\pi^*=(x_0^*,\bx_1^*\dots,\bx^*_T)$ and   an optimal stopping time $\tau^*\in\cT$.
In the risk neutral case when $\varrho_{t|\F_{t}}=\bbe_{|\F_{t}}$, for a given (fixed) policy $\pi\in \Pi$,     these problems  become  the classical problem~\eqref{stop-1}  of stopping time for the process
\begin{equation}
\label{process}
Z_t(\w):=f_t\big(\bx_t(\omega),\w\big).
\end{equation}

  For a given stopping time $\tau\in \cT$ we can write the corresponding dynamic programming equations, of the form~\eqref{dynmar-2},  for the minimization (with respect to $\pi\in \Pi$)  problem~\eqref{stp-max} (cf.\ \cite{Ruszczynski})
	\begin{align}
		V^\tau_T(\x_{T-1},\w)&:=\essinf_{x_T\in\X_t(x_{T-1},\w)}
\one_{\{\tau=T\} } f_T(x_T,\w), \label{eq:1}\\
		V^\tau_t(x_{t-1},\w)&:=\essinf_{x_t\in\X_t(x_{t-1},\w)}
		\one_{\{\tau=t\} } f_t(x_t,\w)+ \varrho_{t|\F_t}\big(V^\tau_{t+1}(x_t,\w)\big), \label{eq:2}
	\end{align}
$ t=1,\dots,T-1$, and the  first stage problem at $t=0$ is (note that
$\varrho_{0|\F_0}=\varrho_0$)
\begin{equation*}\label{first}
	\min_{x_0\in \X_0} f_0(x_0)+\varrho_0  \big(V^\tau_{1}(x_0,\w)\big).
\end{equation*}

\subsubsection{The  min-max   problem}
\label{sec-min-max}
Let us consider the min-max  (minimization with respect to $\pi\in \Pi$~-- maximization with respect to $\tau\in \cT$)  variant of  problem~\eqref{stp-max}.
For a fixed policy $\pi=(x_0,\bx_1,\dots,\bx_T)\in \Pi$  we need to solve the optimal stopping time  problem
\begin{equation}\label{minopt}
 \max_{\tau\in \cT} \left\{\rho_{0,T}(Z_\tau)= \one_{\{\tau=0\}}Z_0+\varrho_{0|\F_0}\big(\one_{\{\tau=1\}}Z_1+
 \cdots
 +\varrho_{T-1|\F_{T-1}}(\one_{\{\tau=T\}}
Z_T) \big)\right\},
\end{equation}
with  $Z_t$ given in~\eqref{process}.    The following  is a natural extension of the classical results in the risk neutral case   to the considered risk averse setting~\eqref{minopt}, e.g.,   \citet[Section 2.2]{ShiryaevStopping}.

\begin{definition}[Snell envelope]
	Let $(Z_0,\dots,Z_T)\in \Z_{0,T}$ be a stochastic process. The Snell envelope  (associated with functional $\rho_{0,T}$)   is the stochastic process
	\begin{align}
		E_T&:= Z_T,\nonumber\\
		E_t&:=  Z_t\vee  \varrho_{t|\F_t}(E_{t+1}),\quad t=0,\dots,T-1,	\label{eq:SnellZ}
	\end{align}
	defined in backwards recursive way.
\end{definition}

For $m=0,\dots,T$,  consider $\cT_m:=\{\tau\in \cT\colon\tau\ge m\}$,
the optimization problem
\begin{equation}\label{optvalue}
\max_{\tau\in \cT_m} \rho_{0,T}(Z_\tau),
\end{equation}
and
\begin{equation}\label{optsol}
	\tau^*_m(\w):=\min\{t\colon E_t(\w)=Z_t(\w),\;m\le t\le T\},\;\w\in \O.
\end{equation}
Denote by $\cv_m$ the optimal value of the problem~\eqref{optvalue}.
Of course, for $m=0$, the problem~\eqref{optvalue} coincides with problem~\eqref{minopt} and   $\cv_0$ is the optimal value of problem~\eqref{minopt}. Note that by the recursive property~\eqref{addrec} we have that\footnote{By the definition    $\rho_{T,T}(Z_T)\equiv Z_T$.}
$\rho_{0,T}(Z_\tau)=\rho_{0,m}(\rho_{m,T}(Z_\tau))$, $m=1,\dots,T$.

The following assumption was used by several authors under different names (see, for example, \cite{Cheridito2009}, where it is called \emph{local property} and references therein):
\begin{equation}\label{local}
	\varrho_{t|\F_t}(\one_A\cdot Z)=\one_A\cdot\varrho_{t|\F_t}(Z),\quad\text{ for all }A\in\F_t,\;t=0,\dots,T-1.
\end{equation}
For   coherent  mappings $\varrho_{t|\F_t}$  it always holds   (cf.\ \citet[Theorem~6.70]{RuszczynskiShapiro2009-2}).

The following can be compared with classical results in the risk neutral case (e.g.,\ \citet[Theorem 1]{BinPes2008}).

\begin{theorem}
\label{thm:VerificationZ}
	Let $\varrho_{t|\F_t}\colon\Z_{t+1}\to\Z_{t}$, $t=0,\dots,T-1$, be  monotone   translation equivariant   mappings  possessing property~\eqref{local} and $\rho_{s,t}$, $0\le s<t\le T$,  be the corresponding nested functionals  defined in~\eqref{stopp-2}.  Then  for $(Z_0,\dots,Z_T)\in \Z_{0,T}$  the following holds:
	\begin{enumerate}[nolistsep, noitemsep]
		\item [{\rm (i)}] for $m=0,\dots,T$, and $\tau^*_m$ defined in~\eqref{optsol},
		\begin{align*}
			E_m&\succeq \rho_{m,T}(Z_\tau),\;\forall \tau\in\cT_m,\\
			E_m&= \rho_{m,T}(Z_{\tau^*_m}),
		\end{align*}
		\item [{\rm (ii)}] the stopping  time $\tau^*_m$  is optimal for the problem~\eqref{optvalue},
		\item [{\rm (iii)}] \label{enu:ii} if $\hat{\tau}_m$ is an optimal stopping time for the problem~\eqref{optvalue}, then
		$\hat{\tau}_m\succeq \tau^*_m$,
		\item [{\rm (iv)}] $\cv_m=\varrho_{0,m}(E_m)$,  $m=1,\dots,T$, and $\cv_0=E_0$.
	\end{enumerate}
\end{theorem}

\begin{proof}
We use induction in $m$ going backwards in time.
Recall that $E_T=Z_T$ and hence  the assertions  follow  for $m=T$.
Now let $m=T-1$ and $\tau\in \cT_{T-1}$.
 Since $ \rho_{T-1,T} =\varrho_{T-1|\F_{T-1}}$, by using the translation equivariance and property~\eqref{local},   we can write
\begin{align*}
	\rho_{T-1,T}(Z_\tau) &= \varrho_{T-1|\F_{T-1}}\left(\one_{\{\tau=T-1\}}Z_{T-1}+\one_{\{\tau=T\}}Z_{T}\right)\\
	&=\one_{\{\tau=T-1\}}Z_{T-1}+\one_{\{\tau=T\}}
	\varrho_{T-1|\F_{T-1}}(Z_{T}).
\end{align*}
 We have that $\O$ is the union of the disjoint sets $\O^\tau_{T-1}$ and $ \O^\tau_T$ (defined in~\eqref{setomega}), and hence (recall that $E_T=Z_T$)
 \begin{equation}
 \label{eqder-2}
 \one_{\{\tau=T-1\}}Z_{T-1}+\one_{\{\tau=T\}}\varrho_{T-1|\F_{T-1}}(Z_{T})
 \preceq   \max\{Z_{T-1}, \varrho_{T-1|\F_{T-1}}(E_T)\}= E_{T-1}.
 \end{equation}
It follows that
$\rho_{T-1,T}(Z_\tau) \preceq   E_{T-1}.$

Conditional on the event $\{Z_{T-1} \ge \varrho_{T-1|\F_{T-1}}(E_T)\}$
we have:
 $E_{T-1} =Z_{T-1} $ and $\tau^*_{T-1} =T-1$, and
 \[
 \rho_{T-1,T}(Z_{\tau^*_{T-1}})=\varrho_{T-1|\F_{T-1}}(Z_{T-1})=Z_{T-1}=E_{T-1},
 \]
and $\tau^*_{T-1}$ is optimal for the corresponding problem~\eqref{optvalue}.
 Otherwise conditional on  the event $\{Z_{T-1} < \varrho_{T-1|\F_{T-1}}(Z_T)\}$, we have that
$E_{T-1}=\varrho_{T-1|\F_{T-1}}(Z_T)$, and $\tau^*_{T-1}=T$ is optimal for   the corresponding problem~\eqref{optvalue}. In both cases the assertion (iii) also holds.

Now let  $m=T-2$ and
$\tau\in \cT_{T-2}$.   We have that $\rho_{T-2,T}(\cdot)=\varrho_{T-2|\F_{T-2}}\big(\varrho_{T-1|\F_{T-1}}(\cdot)\big)$ and
\begin{align}\nonumber
	\rho_{T-2,T}(Z_\tau)&= \varrho_{T-2|\F_{T-2}}\left(\varrho_{T-1|\F_{T-1}}(\one_{\{\tau=T-2\}}Z_{T-2}+
	\one_{\{\tau\ge T-1\}}Z_{\tau})\right)\\
\nonumber
	&=  \one_{\{\tau=T-2\}}Z_{T-2}+\varrho_{T-2|\F_{T-2}}(\one_{\{\tau=T-1\}}Z_{T-1}+
	\one_{\{\tau=T\}}\varrho_{T-1|\F_{T-1}}(Z_{T}))\\
\nonumber
	&=  \one_{\{\tau=T-2\}}Z_{T-2}+\one_{\{\tau>T-2\}}\varrho_{T-2|\F_{T-2}}
	(\one_{\{\tau=T-1\}}Z_{T-1}+
	\one_{\{\tau=T\}}\varrho_{T-1|\F_{T-1}}(Z_{T})),
\end{align}
where the last equation holds  since
$\one_{\{\tau>T-2\}}\one_{\{\tau=T-1\}}=\one_{\{\tau=T-1\}}$ and
$\one_{\{\tau>T-2\}}\one_{\{\tau=T\}}=\one_{\{\tau=T\}}$ and by~\eqref{local}.
Then  by~\eqref{eqder-2} and monotonicity of  $\varrho_{T-2|\F_{T-2}}$ we obtain
\begin{align*}
	\rho_{T-2,T}(Z_\tau)&
	\preceq \one_{\{\tau=T-2\}}Z_{T-2}+\one_{\{\tau>T-2\}}
	\varrho_{T-2|\F_{T-2}}(E_{T-1})\\
	& \preceq  \max\{Z_{T-2}, \varrho_{T-2|\F_{T-2}}(E_{T-1})\}
	=  E_{T-2}.
\end{align*}

Conditional on   $\{Z_{T-2} \ge \varrho_{T-2|\F_{T-2}}(E_{T-1})\}$,
  we have that $E_{T-2} =Z_{T-2} $ and $\tau^*_{T-2}  =T-2$, and
 \[
 \rho_{T-2,T}(Z_{\tau^*_{T-2}})=\varrho_{T-2|\F_{T-2}}
 (\varrho_{T-1|\F_{T-1}}(Z_{T-2}))=Z_{T-2}=E_{T-2},
 \]
and $\tau^*_{T-2}$ is optimal for the corresponding problem~\eqref{optvalue}.
Otherwise conditional on the event   $\{Z_{T-2} < \varrho_{T-2|\F_{T-2}}(E_{T-1})\}$, we have    $E_{T-2} =\varrho_{T-2|\F_{T-2}}(E_{T-1})$ and $\tau^*_{T-2}\ge T-1$.
Conditioning further on $\{Z_{T-1} < \varrho_{T-1|\F_{T-1}}(E_T)\}$ we have that $E_{T-1}=\varrho_{T-1|\F_{T-1}}(Z_T)$, and $\tau^*_{T-2}=T$ is optimal for   the corresponding problem~\eqref{optvalue}. Otherwise conditional further on
$\{Z_{T-1} \ge \varrho_{T-1|\F_{T-1}}(E_T)\}$ we have that $\tau^*_{T-2}=T-1$ and the assertions are verified.

The assertion follows by going backwards in time for $m=T-3,\dots$.
\end{proof}

Note that it follows by~\eqref{eq:SnellZ} that  $Z_t(\w)\le E_t(\w)$  for all $t=0,\dots,T$ and a.e. $\w\in \O$.
We have that
\begin{equation}\label{optstop}
 \tau^*_0(\w)=\min\{t\colon Z_t(\w)\ge E_t(\w), \;t=0,\dots,T\}
\end{equation}
 is an optimal solution of problem~\eqref{minopt},  and the optimal value of problem~\eqref{minopt} is equal to $E_0$.  That is, going forward the optimal stopping time $\tau^*_0$ stops at the first time  $Z_t=E_t$.
In particular it stops
at $t=0$ if $Z_0\ge  E_0$ (recall that $Z_0$ and $E_0$ are deterministic), i.e., iff  $Z_0\ge \varrho_{0|\F_0}(E_1)$; and it stops at $t=T$ iff $Z_t< E_t$  for $t=0,\dots,T-1$.
As in the risk neutral case the time consistency (Bellman's principle) is ensured  here by the decomposable structure of the considered nested risk measure. That is, if it  was not optimal to stop within the time set
$\{0,\dots,m-1\}$, then starting the observation at time $t=m$  and being based on the information
$\F_m$ (i.e., conditional on $\F_m$), the same stopping rule is still optimal for the problem~\eqref{optvalue}.

Consider  (compare with~\eqref{onestep-1})
\begin{equation}\label{onestep-rec}
 \R_{s,s+1}(Z_s,Z_{s+1}):= Z_s\vee\varrho_{s|\F_s}(Z_{s+1}),
 \quad s=0,\dots,T-1.
\end{equation}
Then we can write $E_t$ in the following recursive  form (compare with~\eqref{decom-1})
\begin{equation*}\label{sne-3}
	E_t=\R_{t,t+1}\Big(Z_t, \R_{t+1,t+2}\big(Z_{t+1},\cdots, \R_{T-1,T}(Z_{T-1}, Z_T)\big)\Big),\quad t=0,\dots T-1.
\end{equation*}
Consequently (recall that $\cv_0=E_0$) problem~\eqref{stp-max}
can be written in the form~\eqref{refm-2} with   $\R:=\R_{0,T}$,  where $\R_{0,T}$  is  given in the nested form discussed in Section \ref{sec:decomp}, $\R_{s,s+1}$   defined in~\eqref{onestep-rec}, and with
the respective  dynamic programming equations  of  the   form~\eqref{dyn-stop1}. For an optimal policy $\bar{\pi}=\{\bar{x}_0,\bar{\bx}_1,\dots,\bar{\bx}_{T}\} $,  the  Snell envelop of the corresponding stopping time problem is  $E_t(\w)= V_t(\bar{\bx}_{t-1},\w)$, $t=1,\dots,T$, where $V_t(\cdot,\w)$ is the value function defined by the dynamic equations~\eqref{dyn-stop1}.

\begin{remark}
\label{rem-conv}
As it was  already  mentioned in Remark \ref{rem-max},  there is a delicate issue of preserving convexity of the considered optimization problems. Suppose that the functionals $\varrho_{t|\F_t}$  are convex.  Together with the assumed monotonicity of $\varrho_{t|\F_t}$ and
 since maximum of two convex functions is convex,   convexity of the respective  value functions $V_t(\cdot,\w)$ is implied   by convexity of the objective functions $f_t(\cdot,\w)$  if, for example,  the feasibility constraints are linear of the form
\begin{equation}\label{lincon}
 \X_t(x_{t-1},\w):=\{x_t\ge 0\colon B_t(\w)x_{t-1}+A_t(\w)x_t=b_t(\w)\}.
\end{equation}
On the other hand if $\varrho_{t|\F_t}$  are concave, then convexity of
$V_t(\cdot,\w)$ is not guaranteed.
\end{remark}

\subsubsection{The min-min   problem}
\label{sec-min-min}
Consider the min-min  (minimization with respect to $\pi\in \Pi$ and  $\tau\in \cT$)  variant of  problem~\eqref{stp-max}. In that case for a   fixed policy $\pi\in \Pi$  we need to solve the optimal stopping time  problem
\begin{equation}\label{minopt-2}
 \min_{\tau\in \cT} \rho_{0,T}(Z_\tau).
\end{equation}
Here the    one step mappings $\R_{s,s+1}$,   used in  construction of the corresponding preference system (discussed in Section \ref{sec:decomp}), take the form
\begin{equation}\label{minr}
  \R_{s,s+1}(Z_s,Z_{s+1}):= Z_s\wedge  \varrho_{s|\F_s}(Z_{s+1}),\;
  s=0,\dots T-1.
\end{equation}
In contrast to~\eqref{onestep-rec} considered above, the mappings
$\R_{s,s+1}$ defined in~\eqref{minr} do not preserve convexity of $\varrho_{s|\F_s}$. The corresponding  dynamic programming equations~\eqref{dynmar-2}--\eqref{dynmar-3} (and
\eqref{suff0}--\eqref{suffT})  apply here as well and
are somewhat   simpler  as the essential infimum and the minimum can be interchanged.

\subsubsection{Supermartingales and delayed stopping}
Recall that a sequence of random variables   $\{X_t\}_{0\le t\le T}$ is said to be  \emph{supermartingale} relative to the filtration $\cF$,  if  $X_t\succeq \bbe_{|\F_t}(X_{t+1})$,   $t=0,\dots,T-1$.
By analogy we say that the sequence $X_t\in \Z_t$ is $\cP$-supermartingale, with respect to the collection of  mappings $\cP=\{  \varrho_{t|\F_t}\}_{t=0,\dots,T-1}$,  if
\begin{equation}\label{mart-1}
	X_t\succeq   \varrho_{t|\F_t}(X_{t+1}),\;t=0,\dots,T-1.
\end{equation}
It follows by the definition~\eqref{eq:SnellZ} that the  Snell envelope sequence $\{E_t\}_{0\le t\le T}$ is $\cP$-super\-martingale. It also follows from
\eqref{eq:SnellZ} that $E_t\succeq Z_t$, i.e., $E_t$ dominates $Z_t$.
We have that $\{E_t\}_{0\le t\le T}$ is the smallest  $\cP$-supermartingale which dominates the corresponding sequence $\{Z_t\}_{0\le t\le T}$.

\begin{remark}
\label{rem-order}
	Consider two collections $\varrho_{t|\F_t}\colon\Z_{t+1}\to\Z_{t}$ and $\varrho'_{t|\F_t}\colon\Z_{t+1}\to\Z_{t}$,  $t=0,\dots,T-1$, of  monotone   translation equivariant   mappings  possessing property~\eqref{local}, with the respective Snell envelope sequences $E_t$ and $E'_t$
	and stopping times $\tau'_0=\inf\{t\colon E'_t=Z_t\}$  and  $\tau^*_0=\inf\{t\colon E_t=Z_t\}$, defined  for $(Z_0,\dots,Z_T)\in \Z_{0,T}$.
	Suppose that $\varrho_{t|\F_t}(\cdot)\preceq \varrho'_{t|\F_t}(\cdot)$, $t=0,\dots,T-1$.
	It follows then that $ E_t\preceq E^\prime_t$ for $t=0,\dots,T$, and hence  $\tau^*_0\le\tau^\prime_0$. That is, for larger risk mappings the optimal (maximization)  stopping time, defined in~\eqref{optsol},   is delayed.

	For convex  law invariant risk functionals it holds that $\bbe [X]\le \rho(X)$ (e.g., \citet[Corollary~6.52]{RuszczynskiShapiro2009-2}). In that case it follows together with~\eqref{mart-1} that every $\cP$-supermartingale $X_t$, $t=0,\dots,T-1$,  is also a martingale in the usual sense, i.e., with respect to the expectation.
For the concave counterpart $\nu(X)=-\rho(-X)$ the converse inequality $\nu(X)\le \bbe [X]$
follows of course.
\end{remark}

\section{Numerical illustration}\label{sec:numeric}
This section discusses computational approaches to solving   stopping time problems with preference systems of the general form~\eqref{stp-max}.
We illustrate two different approaches based on the pricing of American put options. This stopping time problem is well-known in mathematical finance.

We start by solving the dynamical equations explicitly in the following section and then elaborate on Stochastic Dual Dynamic Programming (SDDP) type algorithm  in Section~\ref{sec:SDDP}. We give numerical examples for optimal stopping in univariate, as well as for multivariate problems (basket options). Further, the examples cover both, the convex and the concave nested (stopping) preference systems.

\subsection{Optimal stopping by solving the dynamic equations}\label{sec:pric}
Like a simple stock, the option is an investment which comes with risk. The risk averse investor is not willing to pay the fair price for risk-prone investments, instead, the potential buyer expects a discount corresponding to his general decline of risk.
As the American put option with strike price $K>0$ can be exercised any time the investor considers the optimal stopping problem
\begin{equation}\label{pr-1}
	\sup_{\tau \in \cT}\rho_{0,T}\big(e^{-r\,\tau}\cdot[K-S_\tau]_+ \big),
\end{equation}
where $r>0$ is a fixed discount rate and $S_t$ is the price of the underlying at time $t$.

The risk neutral investor chooses the expectation, $\rho_{0,T}=\mathbb E$, in~\eqref{pr-1}, this problem is the well-known optimal stopping problem in mathematical finance.
The risk averse investor, in contrast, chooses a preference measure $\rho_{0,T}$ reflecting his personal risk profile.

Note that the situation reverses for the \emph{owner} of a risky asset. This market participant is willing to accept a fee (a commission, like a premium in insurance) to get rid of his risky asset, i.e., to sell the risk.
As well, he will choose a preference functional when computing his personal price of the option, but his personal preference functional reveals opposite preferences than the buyer's functional.

Following the Black-Scholes scheme we consider the geometric random walk process
\begin{equation}\label{pr-2}
	S_t=S_{t-1}\cdot\exp\big(r-\sigma^2/2+\e_t\big), \;t=1,\dots,T,
\end{equation}
in discrete time with 
$\e_t$ being an i.i.d.\ Gaussian white noise process, $\e_t\sim\mathcal N(0,\sigma^2)$.

By considering $S_t$ as the state variable the corresponding dynamic equations   can be written recursively as
\begin{equation}\label{pr-3}
	\V_T(S_T):=[K-S_T]_+
\end{equation}
and
\begin{equation}\label{pr-4}
	\V_t(S_t)=   [K-S_t]_+\vee e^{-r}\cdot\varrho_{t|S_t}\big(\V_{t+1}(S_{t+1})\big),
\end{equation}
for $t=T-1,\dots,1$, where $\varrho$ is a chosen law invariant  coherent  risk measure. In particular we can use the Entropic Value-at-Risk, $\varrho:= \evar^\beta$, where for $\beta>0$,
\begin{equation}\label{pr-6}
	\evar^\beta(Z):=\inf_{u>0}\left\{u^{-1}\left(\beta+\log \bbe[e^{uZ}]\right)\right\}.
\end{equation}
Note that this is a \emph{convex} law invariant  coherent  risk measure, it is the homogeneous version of the risk measure studied in \citet{Schachermayer2009a}. We refer to \cite{AhmadiPichler} for details on this particular risk measure, which allows explicit evaluations for Gaussian random variables.

\begin{figure}
	\subfloat[Regions for the risk averse option holder]
	{\includegraphics[width=0.5\textwidth]{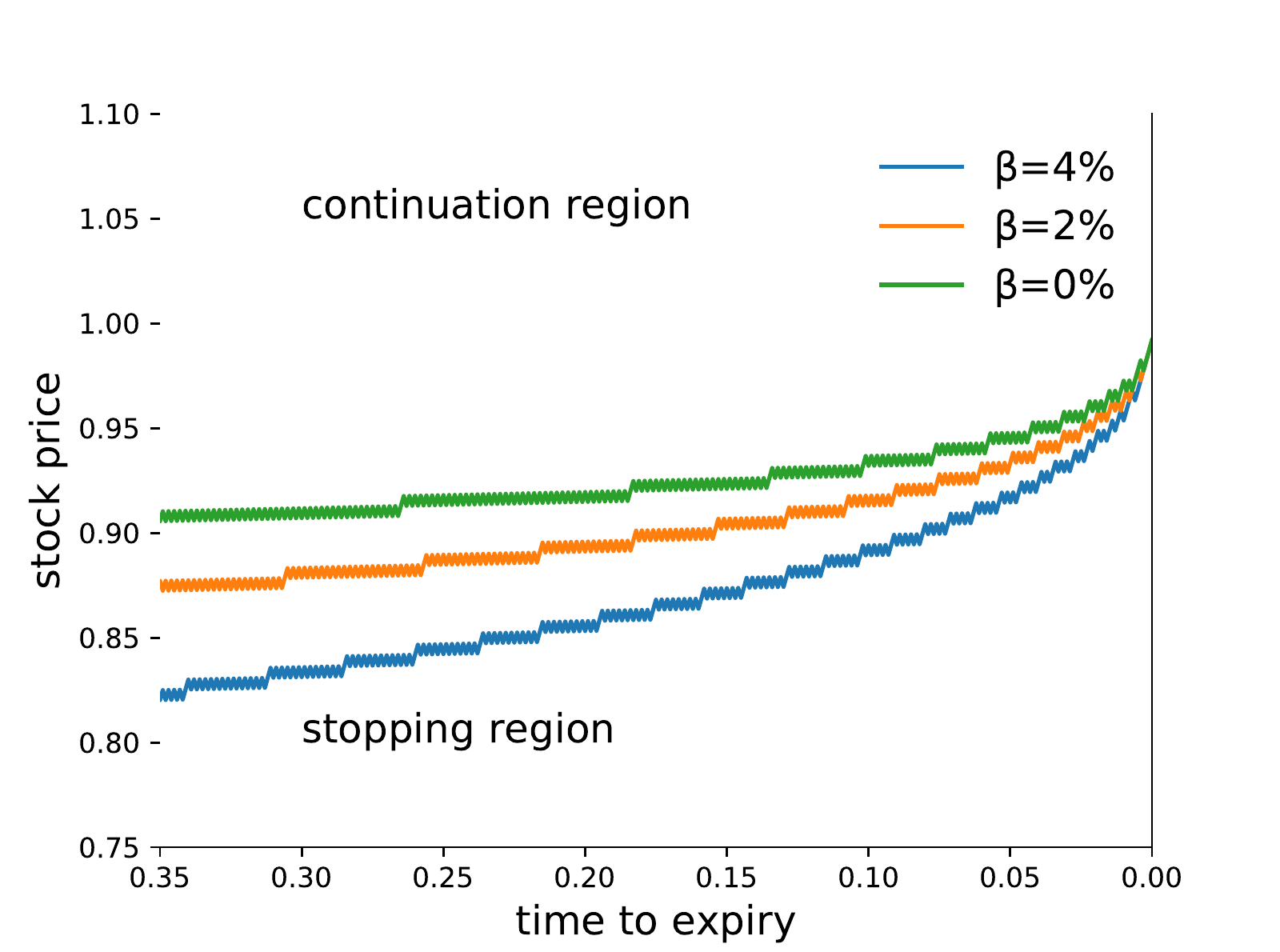}
		\label{fig:11}}
	\hfill
	\subfloat[Regions for the risk averse option buyer]
	{\includegraphics[width=0.5\textwidth]{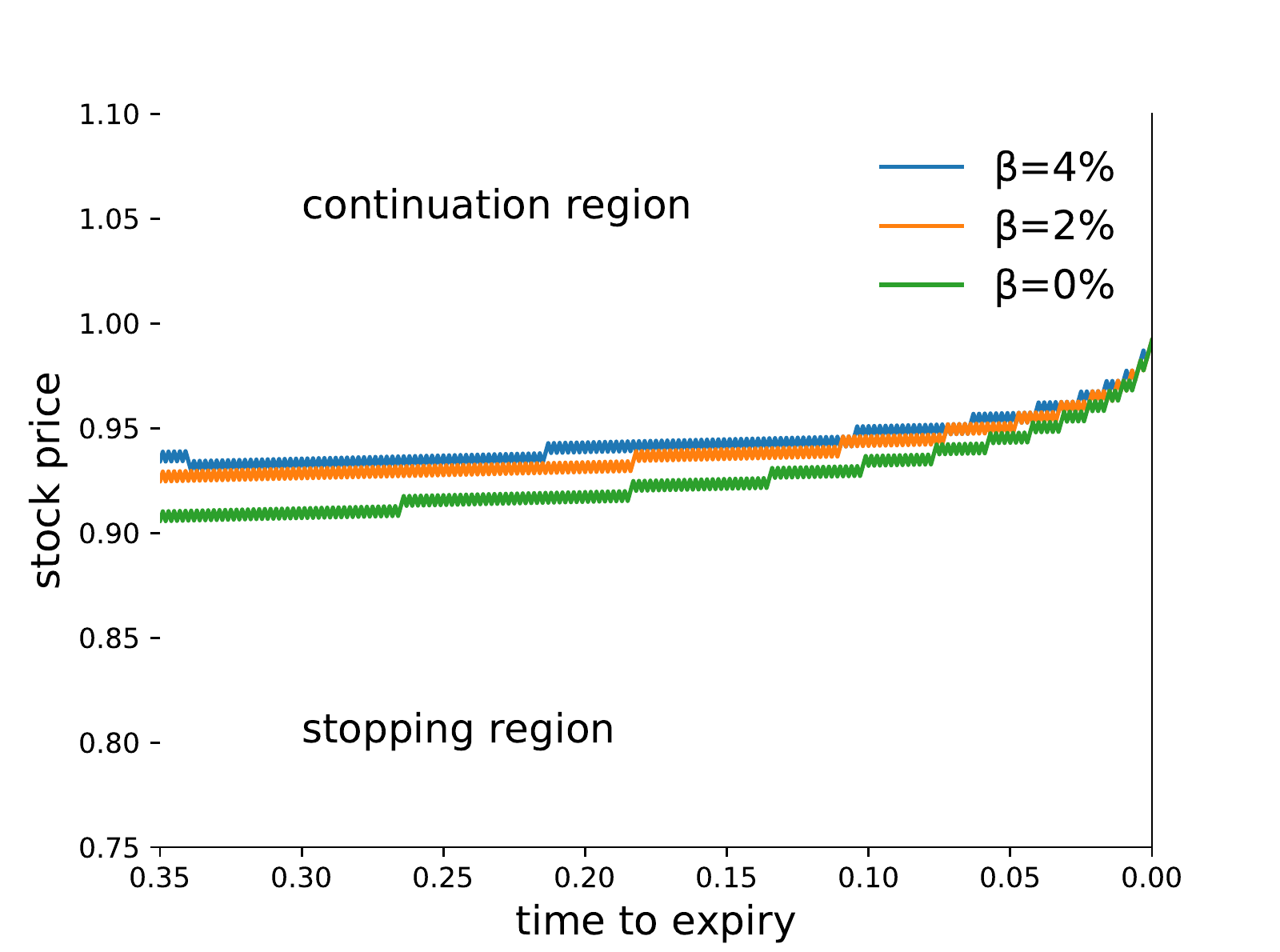}
		\label{fig:12}}
	\caption{\label{fig:1}Stopping and continuation regions for risk averse option traders}
	\label{fig:Prices}
\end{figure}
Figure~\ref{fig:1} displays the decision rules for $\varrho=\EVaR^\beta$ and varying levels $\beta$ of risk aversion. The option is not exercised, as long as the price of the stock stays in the continuation region. Once the price of the stock drops out of the continuation region into the stopping region, then it is optimal to exercise the put option. The optimal stopping time thus is \[\tau=\inf\big\{i\in\{0,\dots,n\}\colon (t_i,S_{t_i})\not\in \text{ continuation region}\big\}.\]
In view of~\eqref{pr-4} the stopping rule can be restated as
\begin{equation}\label{eq:7}
	S_t + \V_t(S_t)\le K,
\end{equation} where $S_t$ is the current price of the stock and $\V_t(S_t)$ the actual and updated price of the corresponding put option. The rule~\eqref{eq:7} is known as \emph{fugit} in mathematical finance.

It holds as well that $-\varrho(-Y)\le \varrho(Y)$\footnote{Indeed, $0=\varrho(0)\le\varrho(Y)+\varrho(-Y)$.} and Figure~\ref{fig:12} displays the continuation and stopping region for the respective  concave functional $\nu(Y):=-\EVaR^\beta(-Y)$, the solution of the problem
\[\sup_{\tau\in\cT} \nu_{0,T}\big(e^{-r\,\tau}\cdot[K-S_\tau]_+\big);\]
this plot describes the trading regions for the risk averse option buyer. In the risk neutral case, $\beta=0$, the regions are notably identical as the expectation is linear and thus $\mathbb E [Y]=-\mathbb E[-Y]$.


\subsection{Stochastic Dual Dynamic Programming}\label{sec:SDDP}
The  SDDP algorithm was introduced in \cite{per1991} and extended to a risk averse setting in \cite{shap2011a}. For a discussion of the SDDP method we can refer to \citet[Section~5.10] {RuszczynskiShapiro2009-2} and references therein.

\subsubsection{Univariate SDDP}
To study the performance of the SDDP approach to solving  optimal stopping problems,   we remove nonlinear influences and consider  the arithmetic version of the reference problem~\eqref{pr-1}, which is
\[ \sup_{\tau \in \cT}\rho_{0,T}\big([K-S_\tau]_+ -r\tau\big)
\]
(\cite{SchachermayerBachelier} discuss the differences of these models).
We use risk measure  $\varrho := (1 - \lambda) \mathbb{E} + \lambda \avr_\alpha$ for some $\lambda\in [0, 1]$ and $\alpha\in (0, 1)$ and consider  $T = 25$ stages.
 The SDDP algorithm first discretizes the random variables $\e_t$. If $\e_t$ takes the values $\pm \sigma\,S_0$ with probability~$\nicefrac12$ each, then the dynamic programming equations reduce to the binomial option pricing model. Nevertheless, the SDDP algorithm can handle the general situation where $\e_t$ follows an arbitrary distribution, as long as it is possible to obtain samples from such distribution. In the following experiments we let $\e_t\sim \N(0, \sigma^2S_0^2)$.
In order to solve the problem numerically we discretize the (continuous) distribution by randomly generating $N$ realizations at every stage. We refer to the obtained approximation as the \emph{discretized problem}.
The following numerical experiments use $N=100$  discretized points for each time period.

We assess the quality of the SDDP algorithm on two aspects, namely the statistical properties of the approximations of dynamic programming equations formed under discretization, and the efficiency of the algorithm in solving the discretized problem.
Since the problem is univariate (i.e., a single stock price $S_t$), it is possible to solve the discretized   problem quite accurately simply by discretization of the state variables.  So the main purpose of the following exercise is to verify efficiency of  the SDDP algorithm. After that we investigate a multivariate setting where simple discretization of state variables is not possible.
We first run the algorithm, applied to the discretized problem,  for 1000 iterations and record the lower bounds generated in each iteration.

The   upper bound of the optimal value is constructed by piecewise linear approximations of the value functions $V_t(\cdot)$ (this
could be inefficient in   the multivariate setting).
Note that the SDDP algorithm, applied to the discretized problem, also generates a policy for the original problem with continuous distributions of $\e_t$.  Its value for the original problem can be estimated by generating random sample paths (scenarios)  from the original distributions and averaging the obtained values.
We sample 2000 scenarios from the distributions of $\e_t$ (either the original or the discretized one) and plot the distributions of the corresponding stopping time and the interest discounted profit. To assess how well the discretization  approximates the true dynamic programming equations, we run the SDDP algorithm 30 times and compare the distributions of stopping time and interest discounted profit.

\paragraph{Formulation 1: $\lambda = 0$ (the risk neutral case).}
Figure \ref{fig1} exhibits a typical convergence of lower bounds generated by the SDDP algorithm. In particular, the gap between the upper and the lower bounds is already small after, say, 500 iterations. The algorithm solves the discretized problem  quite well.
Moreover values of the constructed policies for the discretized and original problems are very similar.


\begin{figure}[H]
	\centering
	\includegraphics[scale=0.50]{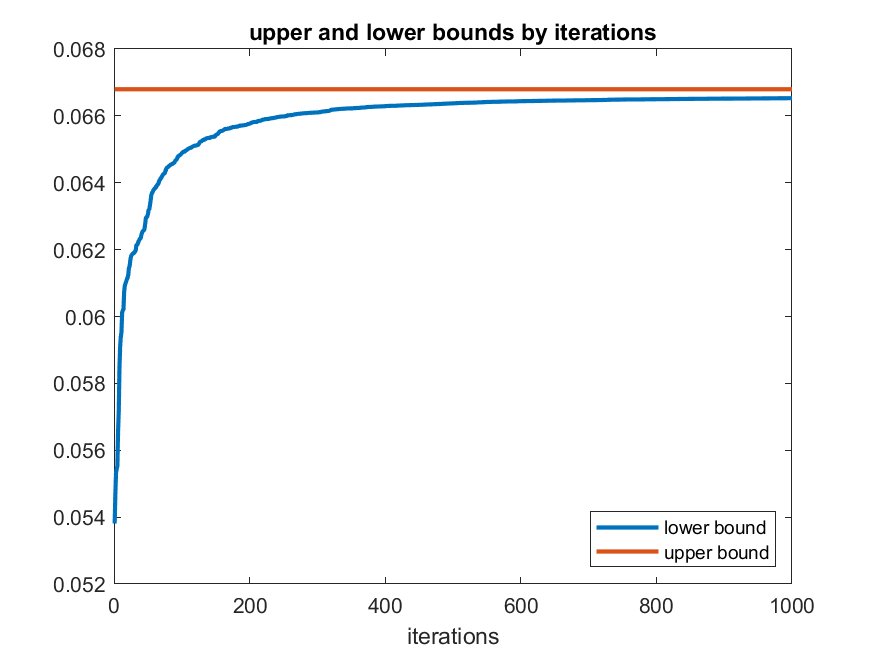}
	\caption{Upper and lower bounds of the optimal objective values (of formulation 1) generated by the SDDP algorithm in 1000 iterations.}
	\label{fig1}
\end{figure}

We run 30 trials of the SDDP algorithm for randomly generated discretizations of the original continuous distributions of $\e_t$.
We sample 2000 scenarios from both the true and the discretized distributions of $\e_t$ and plot the distributions of the corresponding stopping time and the interest discounted profit. We refer to the discretized distribution as the empirical distribution.

Figures~\ref{fig2} and~\ref{fig3} contain plots for typical distributions of interest discounted profits and stopping times generated by the discretized problems  in 30 trials, respectively. In particular, the subplot on the left of each plot corresponds to the scenarios sampled from the discretized distribution of $\e_t$, whereas the subplot on the right corresponds to the original normal distribution.
Although figure \ref{fig3} indicates that  the distributions of the stopping times generated
in different trials could  be different,
the distribution of the profits share a common shape as shown in Figure \ref{fig2}. This indicates that  the optimal objective value of \eqref{pr-1} generated by the discretization  is reasonably accurate when $N = 100$, while optimal stopping times are unstable.

\begin{figure}[H]
	\centering
	\includegraphics[scale=0.55]{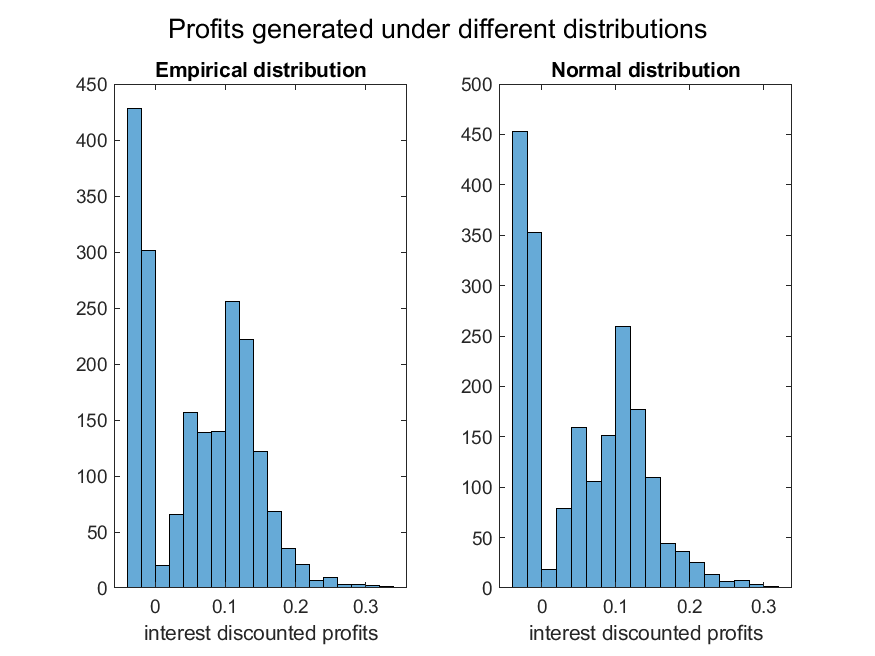}
	\includegraphics[scale=0.55]{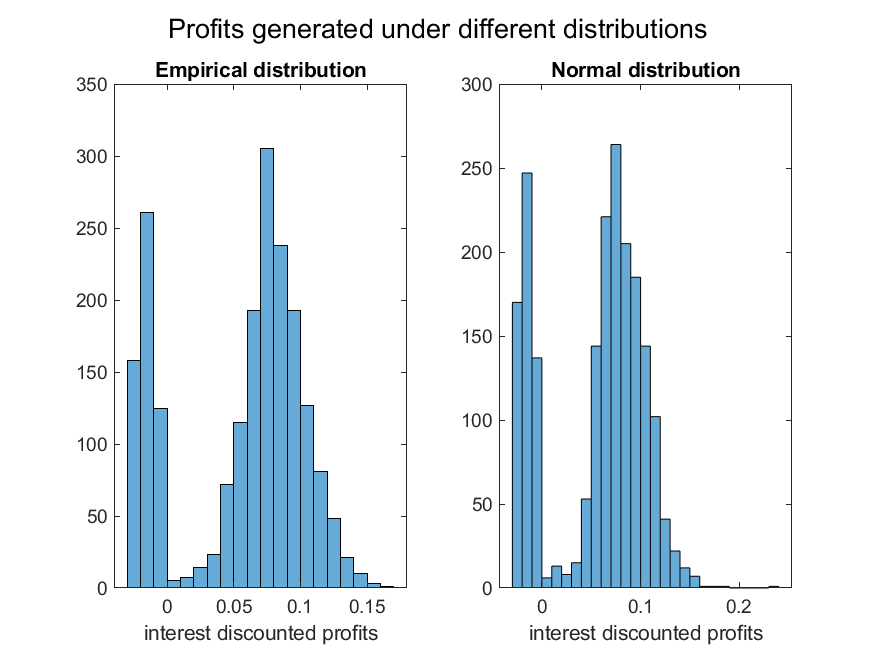}
	\caption{Typical distributions of interest discounted profit (of formulation 1) where the scenarios are sampled from both the empirical and the original normal distribution.}
	\label{fig2}
\end{figure}

\begin{figure}[H]
	\centering
	\includegraphics[scale=0.55]{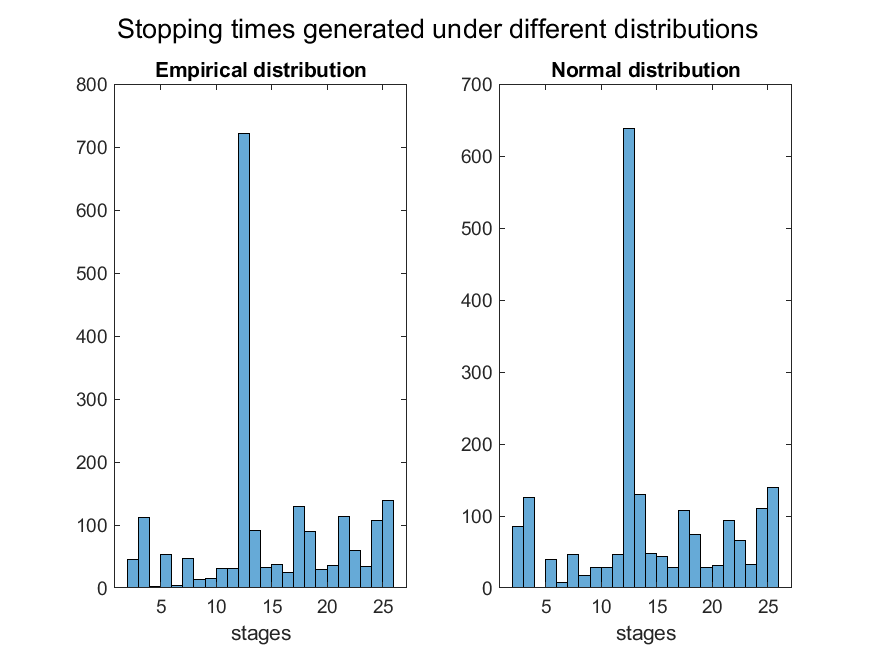}
	\includegraphics[scale=0.55]{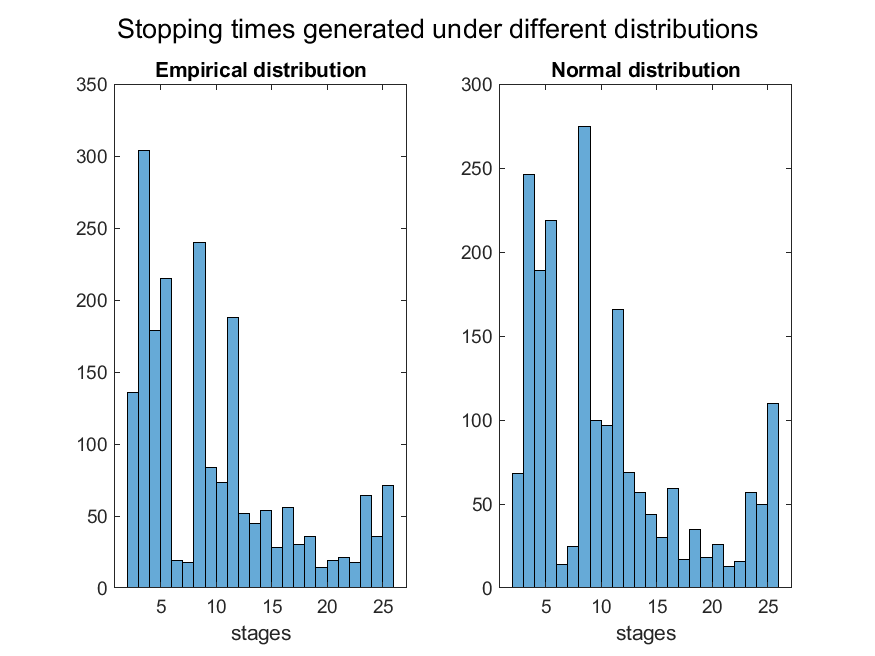}
	\caption{Typical distributions of stopping time (of formulation 1) where the scenarios are sampled from both the empirical and the original normal distribution.}
	\label{fig3}
\end{figure}


\paragraph{Formulation~2: $\lambda = 0.2$ and $\alpha = 0.05$.}
We present similar analysis  as in the risk neutral case. Figure~\ref{fig5} shows the gap between the upper and the lower bounds which is already small after 500 iterations. Moreover, all distributions of stopping times and profits have the same shape (see Figure~\ref{fig6}). As predicted in Remark~\ref{rem-max}, such  formulation hopes that  the uncertainty of the probability distribution works in our favor potentially giving a larger value of the return, thereby delays the stopping time to late stages.

\begin{figure}[H]
	\centering
	\includegraphics[scale=0.50]{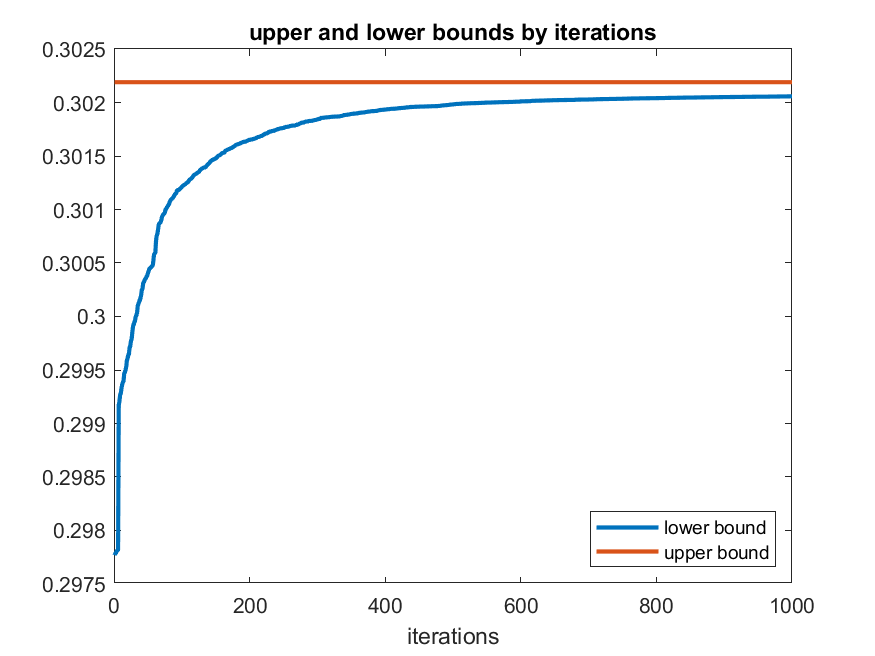}
	\caption{Upper and lower bounds of the optimal objective values (of formulation 2) generated by the SDDP algorithm in 1000 iterations.}
	\label{fig5}
\end{figure}

\begin{figure}[H]
	\centering
	\includegraphics[scale=0.55]{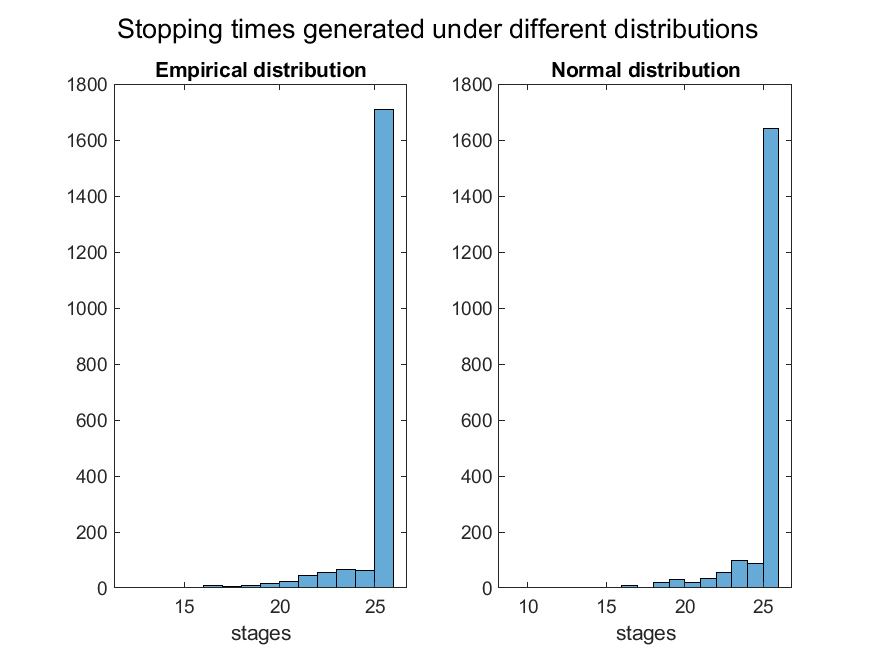}
	\includegraphics[scale=0.55]{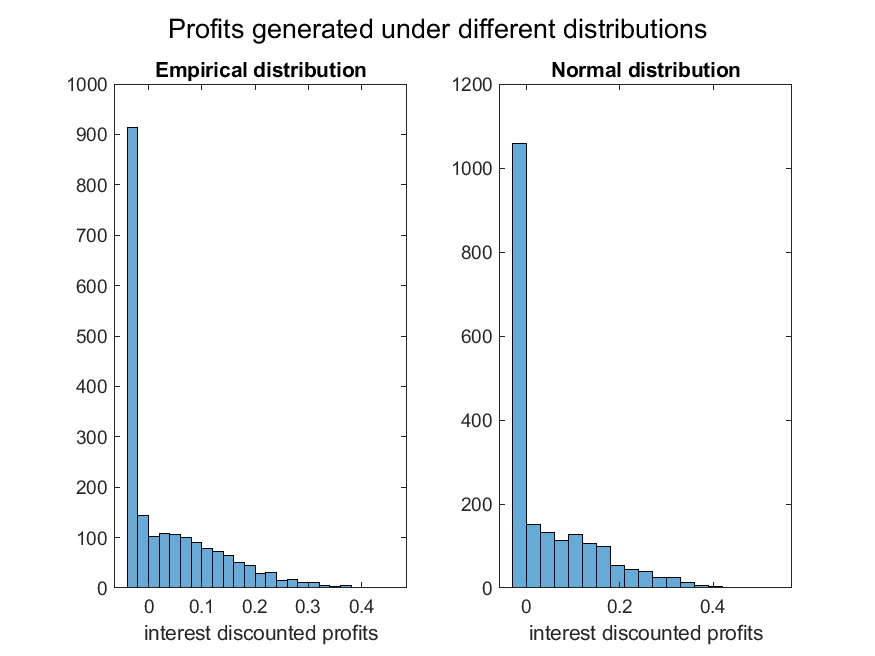}
	\caption{Typical distributions of stopping time and interest discounted profit (of formulation 2) where the scenarios are sampled from both the empirical and the original normal distribution.}
	\label{fig6}
\end{figure}


\subsubsection{Basket options}
We also analyze the SDDP algorithm applied to the  (multivariate)
American basket option pricing problem
\begin{equation} \label{bop}
\begin{array}{ll}
\sup\limits_{\tau \in \cT}\rho_{0,T}\left(\left[\sum_{j\in J} w_j S^j_\tau - K\right]_+ -r\tau\right).
\end{array}
\end{equation}
Here $J$ is an index set of assets with $|J|=5$,   and $T = 25$. Also
$\varrho = (1 - \lambda) \mathbb{E} + \lambda \avr_\alpha$ for some $\lambda\in [0, 1]$ and $\alpha\in (0, 1)$,   $w_j\in \mathbb{R}$ are weights, $K$ is the strike price,  $S^j_t$ is the price of asset $j$ at time $t$.  In particular, suppose $\mu^j = rS^j_0$ and iid $\epsilon_t\sim N(\mathbf{0}, \Sigma)$ for some covariance matrix $\Sigma$ of order $|J|\times |J|$, then $S^j_t$ is a random walk process for each $j\in J$ such that
\begin{equation} \label{udr}
S^j_t = \mu^j + S^j_{t-1} + (\epsilon_t)_j, \;t = 1, \ldots, T.
\end{equation}
The SDDP algorithm for \eqref{bop} proceeds the same way as the one for \eqref{pr-1}, and our analysis of the algorithm is the same as before, i.e., we assess how well the discretization  approximates the true dynamic programming equations and how efficient the algorithm in solving the original problem. As mentioned before, the construction of a deterministic upper bound could be inefficient in the multivariate setting, hence we simply let the algorithm run 1000 iterations.

\paragraph{Formulation 3: basket option with $\lambda = 0$ (the risk neutral case).}
Figure~\ref{fig8} shows that  the lower bounds generated by the SDDP algorithm stabilize after about 500 iterations.
It appears that  the SDDP algorithm solves the discretized problem quite accurately, and we evaluate its performance for the original problem by generating scenarios from the original distribution.
As shown in Figures~\ref{fig9} and~\ref{fig10}, the distributions of policies, namely interest discounted profits and stopping times, vary across trials.

To further understand the relation between the quality of the discretization  and the number of discretized points, we run 30 additional trials with $N = 300$ for each time period, and the results are summarized in Figures~\ref{fig11} and~\ref{fig12}. In short, the distributions of profits now look similar, though the distributions of stopping times still have a moderate variation.

\begin{figure}[H]
	\centering
	\includegraphics[scale=0.50]{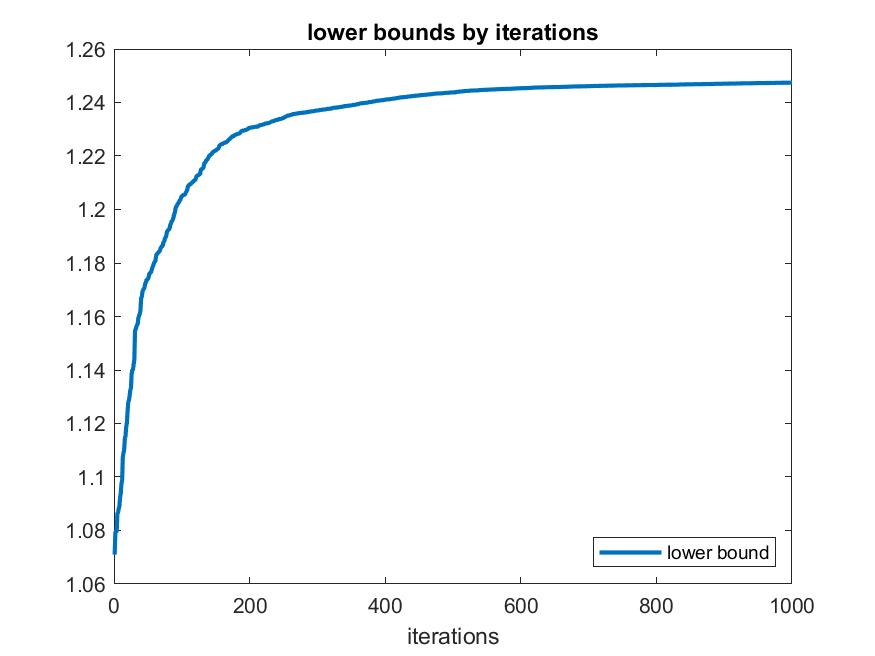}
	\caption{Lower bounds of the optimal objective values (of formulation 3) generated by the SDDP algorithm in 1000 iterations.}
	\label{fig8}
\end{figure}

\begin{figure}[H]
	\centering
	\includegraphics[scale=0.55]{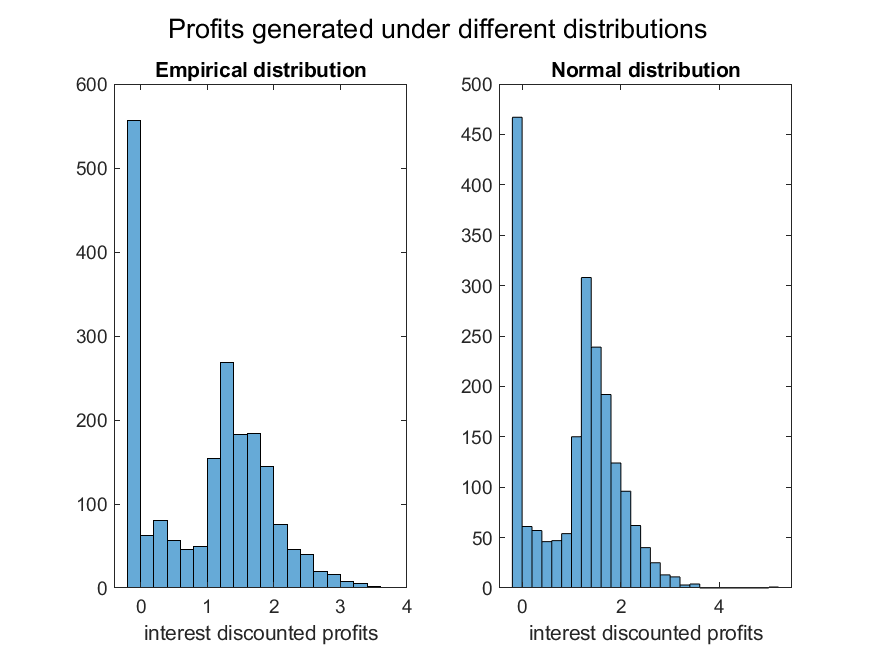}
	\includegraphics[scale=0.55]{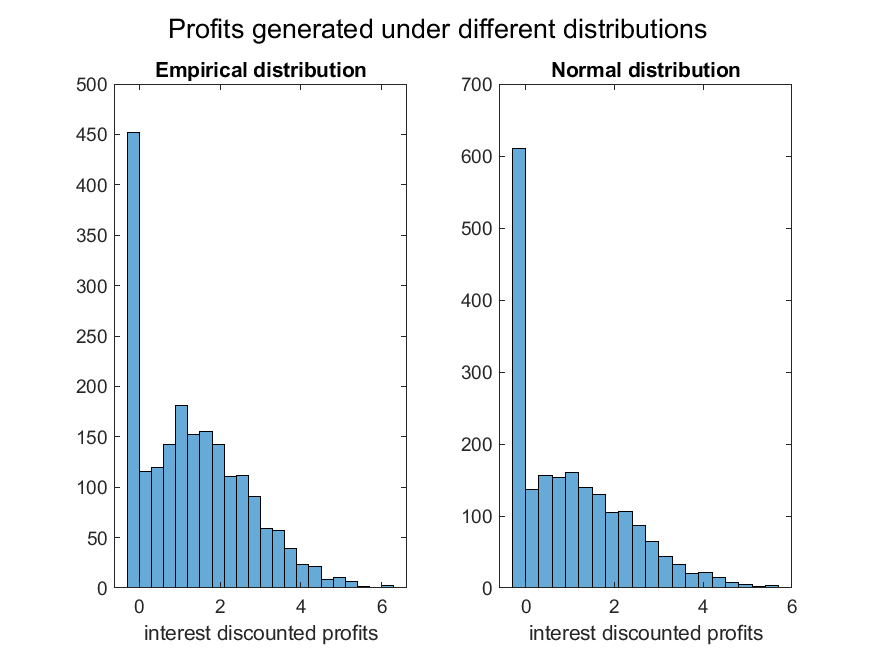}
	\caption{Typical distributions of interest discounted profit (of formulation 3 with $N = 100$) where the scenarios are sampled from both the empirical and the original normal distribution.}
	\label{fig9}
\end{figure}

\begin{figure}[H]
	\centering
	\includegraphics[scale=0.55]{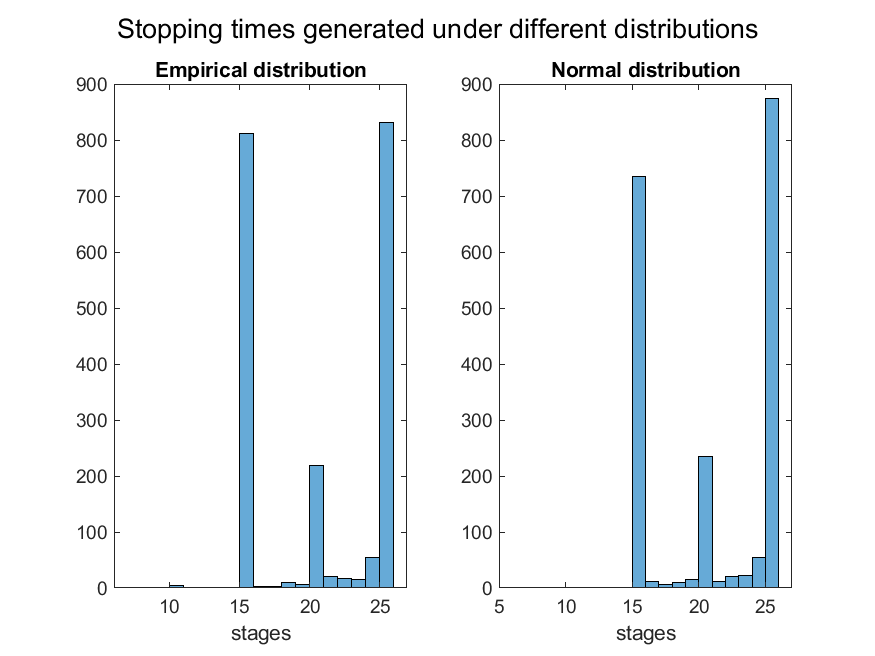}
	\includegraphics[scale=0.55]{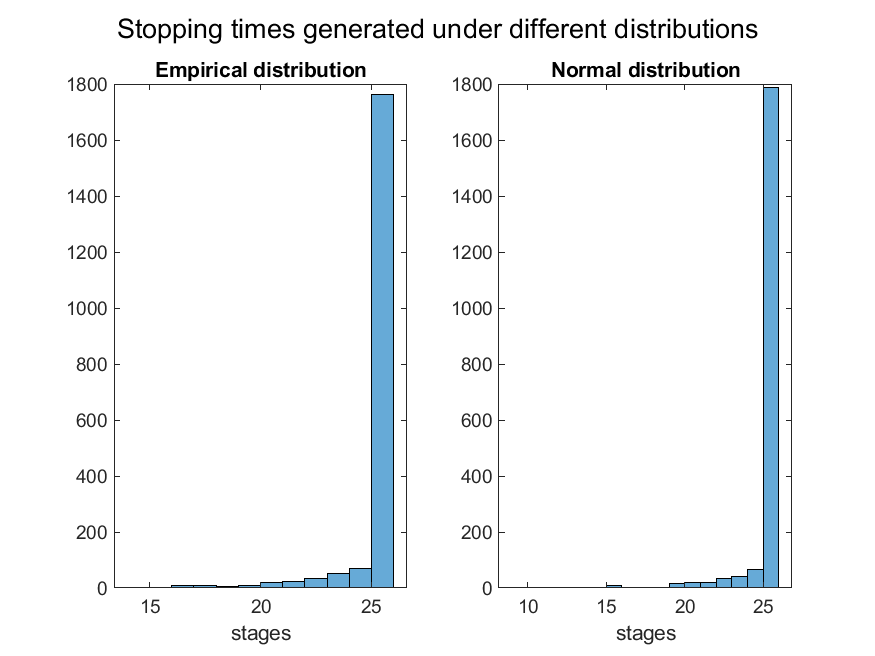}
	\caption{Typical distributions of stopping time (of formulation 3 with $N = 100$) where the scenarios are sampled from both the empirical and the original normal distribution.}
	\label{fig10}
\end{figure}

\begin{figure}[H]
	\centering
	\includegraphics[scale=0.55]{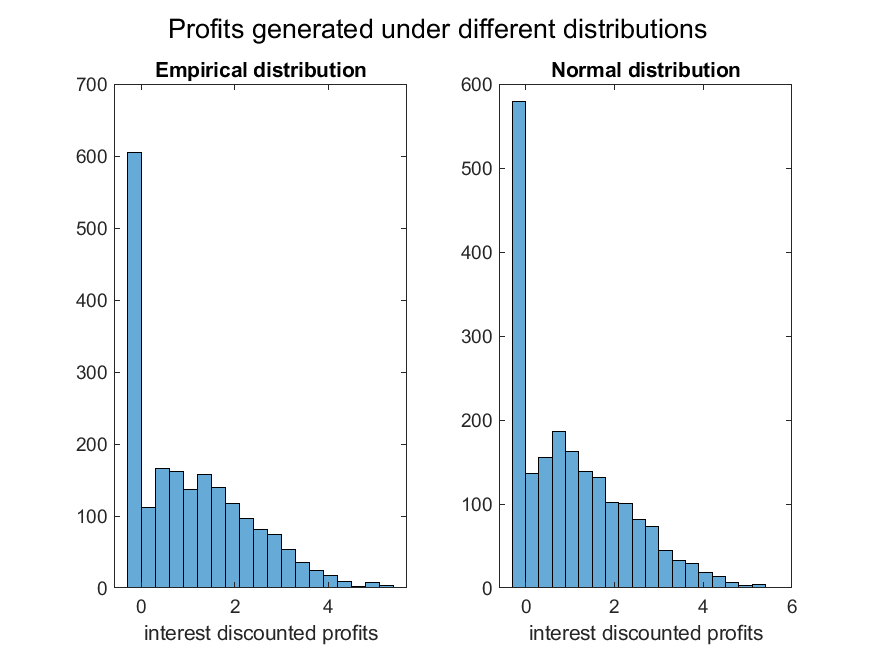}
	\includegraphics[scale=0.55]{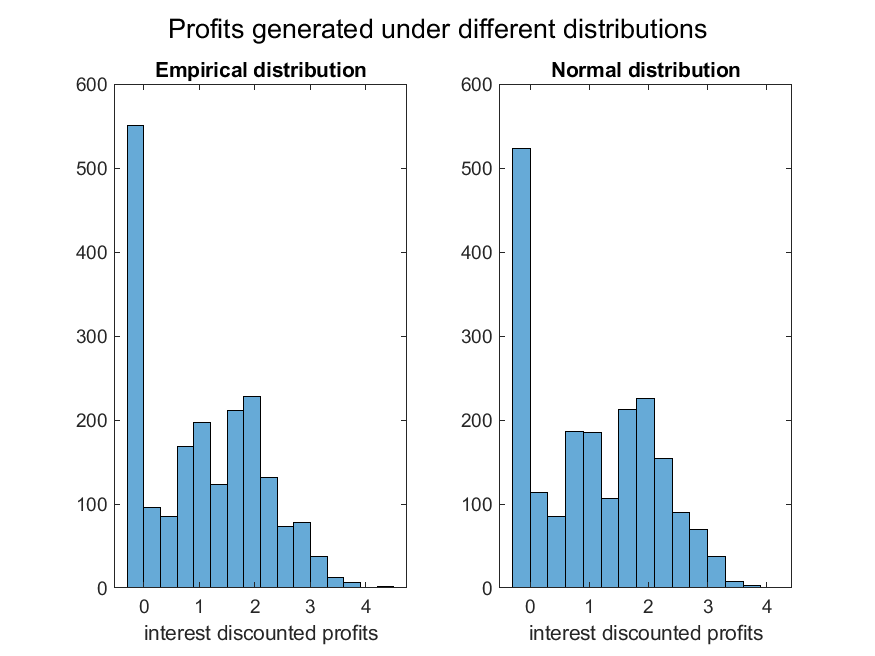}
	\caption{Typical distributions of interest discounted profit (of formulation 3 with $N = 300$) where the scenarios are sampled from both the empirical and the original normal distribution.}
	\label{fig11}
\end{figure}

\begin{figure}[H]
	\centering
	\includegraphics[scale=0.55]{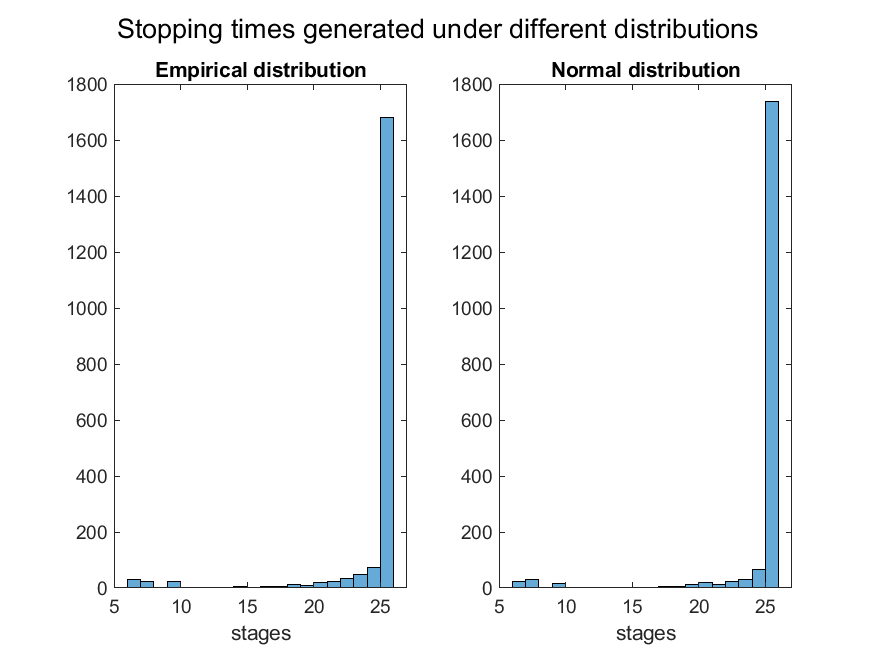}
	\includegraphics[scale=0.55]{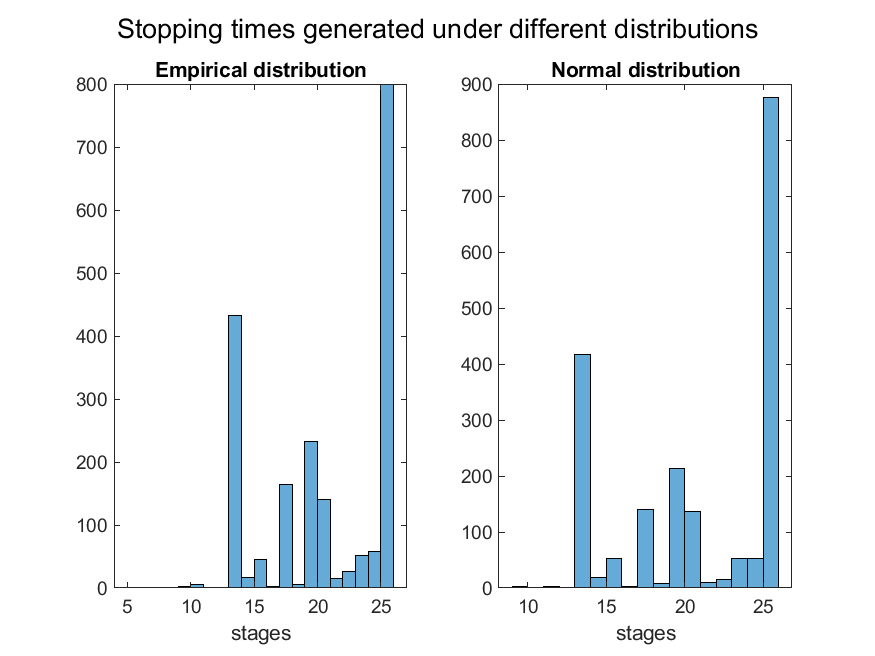}
	\caption{Typical distributions of stopping time (of formulation 3 with $N = 300$) where the scenarios are sampled from both the empirical and the original normal distribution.}
	\label{fig12}
\end{figure}

\paragraph{Formulation 4: basket option with $\lambda = 0.2$ and $\alpha = 0.05$.}  Like the  other formulations, the lower bounds generated by the SDDP algorithm in this case stabilize after about 500 iterations, hence we omit the plot. Similar to formulation 2, all distributions of stopping times and profits have the same shape (see Figure~\ref{fig13}), and the stopping time concentrates on late stages. This indicates that the discretization gives  is a good approximation of the original problem  even when $N = 100$ for each time period.

\begin{figure}[H]
	\centering
	\includegraphics[scale=0.55]{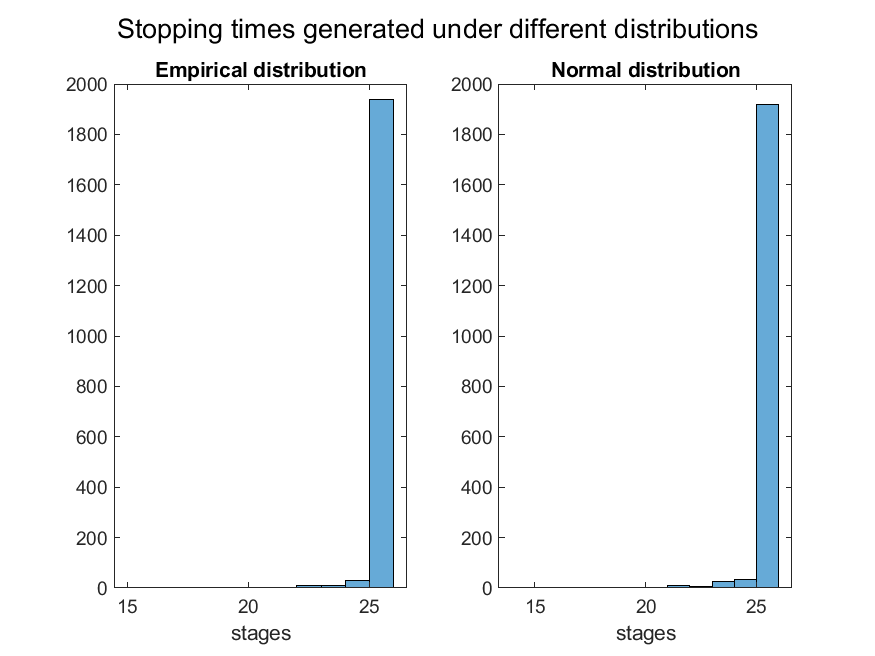}
	\includegraphics[scale=0.55]{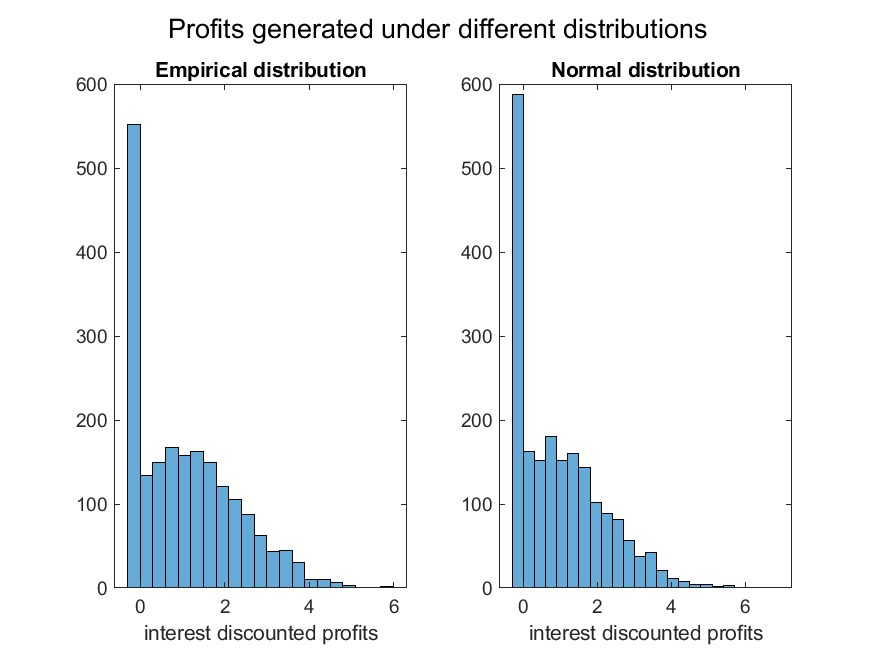}
	\caption{Typical distributions of stopping time and interest discounted profit (of formulation 4) where the scenarios are sampled from both the empirical and the original normal distribution.}
	\label{fig13}
\end{figure}

We next consider the American basket option pricing problem \eqref{bop} with concave risk measure, with  $T = 25$ and $\varrho(Z) = (1 - \lambda) \mathbb{E}(Z) - \lambda \avr_\alpha(-Z)$ for some $\lambda\in [0, 1]$ and $\alpha\in (0, 1)$. This is a risk averse formulation since the concave risk measure shifts more weights to scenarios that reduce the profits. As mentioned in Remark~\ref{rem-max}, the convexity is not preserved when a concave risk measure is used.

Without convexity, the SDDP algorithm cannot be applied directly. We circumvent this by approximating the risk averse formulation by solving a sequence of risk neutral problems. The basic idea is to treat every stage of the risk averse formulation as a saddle point problem like \eqref{rem-stab}. Starting with the discretized distribution that assigns equal probability to each scenario at every stage, we iteratively solve the risk neutral problem and construct a new discretized distribution by shifting more weights to the ``bad" scenarios for the current risk neutral problem. Specifically, in each iteration, we first let the SDDP algorithm solve the discretized  risk neutral problem. Then we sample scenarios from the discretized distribution; each scenario corresponds to a sequence of stock prices $\{S_0^j\}_{j\in J}, \{S_1^j\}_{j\in J}, \ldots, \{S_T^j\}_{j\in J}$ according to the equation \eqref{udr}. For each $t = 1, \ldots, T$, we
consider $V_t(\{S^j_{t - 1}\}_{j\in J}, \hat{\epsilon}_t)$ constructed by the discretized problem, where $\hat{\epsilon}_t$ are the possible realizations of $\epsilon_t$ in the discretized distribution. Consequently we   arrange these values  in the  ascending order.
For each $\hat{\epsilon}_t$, we then count the frequency (i.e., the number of scenarios) that $\underline{V}_t(S_{t - 1}, \hat{\epsilon}_t)$ has one of the $\lceil \alpha N \rceil$-lowest values. The new distribution is constructed based on the frequency such that the realization $\hat{\epsilon}_t$ is assigned the probability
\begin{itemize}
	\item $\frac{1 - \lambda}{N} + \frac{\lambda}{\alpha N}$, if its frequency is among the top $\lfloor \alpha N \rfloor$ frequencies for stage $t$.
	
	\item $\frac{1 - \lambda}{N} + \frac{\lambda (\alpha N - \lfloor\alpha N\rfloor)}{\alpha N}$, if its frequency is the $(\lfloor \alpha N \rfloor + 1)$-highest frequency for stage $t$.
	
	\item $\frac{1 - \lambda}{N}$, otherwise.
\end{itemize}
Note that those probability weights are exactly the ones assigned by $\varrho$ to different outcomes of the random variable $Z$ when $Z$ has only finitely many outcomes.

The sequence of risk neutral problems is not guaranteed to converge to the original problem in general. Nonetheless, this approach works well on the basket option pricing problem we tested. First of all, the sequence of constructed distributions converges. Besides, the policies generated by the discretization,  obtained in the last iteration,  exhibit  risk aversion. Indeed, compared to Figures~\ref{fig10} and~\ref{fig12}, the stopping time shown in Figure~\ref{fig14} has more weights in the early stages, and the distribution of the interest discounted profit (see Figure~\ref{fig14}) has a shorter range, because the higher potential returns are turned down in exchange for less risk.

\begin{figure}[H]
	\centering
	\includegraphics[scale=0.55]{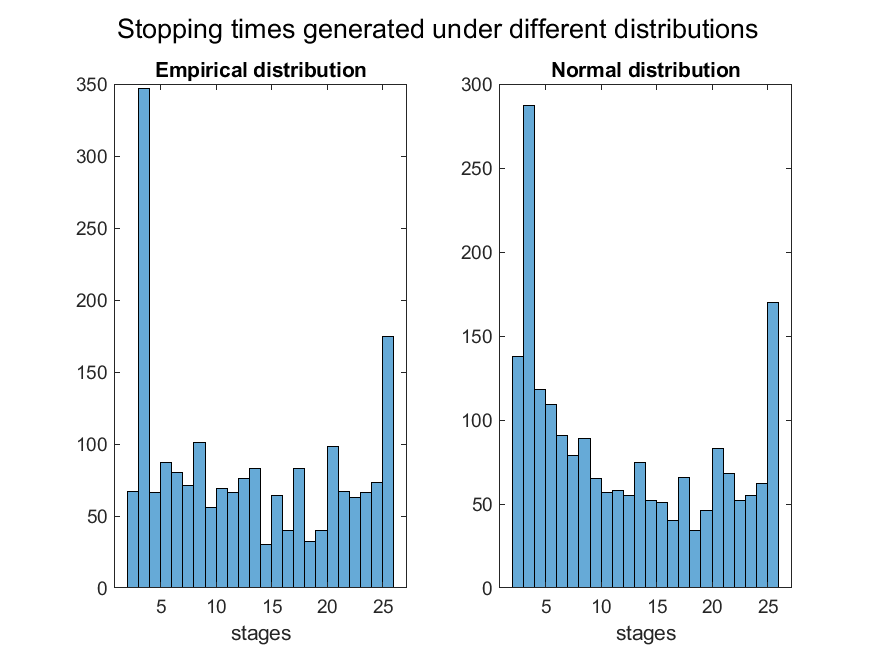}
	\includegraphics[scale=0.55]{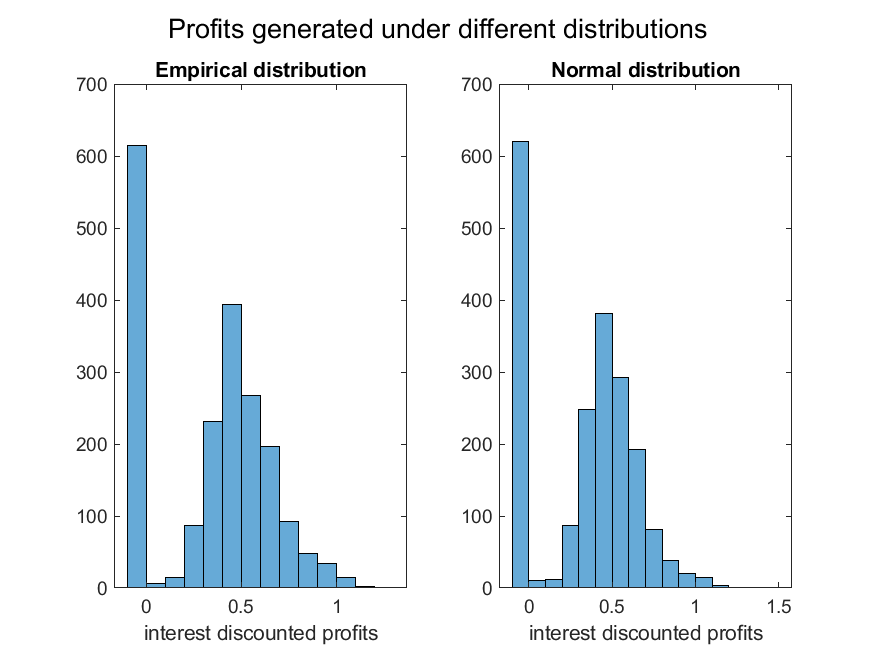}
	\caption{Typical distributions of stopping time and interest discounted profit (of formulation 4) where the scenarios are sampled from both the empirical and the original normal distribution.}
	\label{fig14}
\end{figure}

\bibliographystyle{abbrvnat}
\bibliography{../Literatur/LiteraturAlois,LiteratureAlexander}
\end{document}